\documentclass[12pt]{article}
\usepackage{amsmath,amssymb,amsthm,url,amsfonts,enumerate}
\usepackage{graphicx}
\usepackage{hyperref}
\hypersetup{
   colorlinks   = true, 
   urlcolor     = blue, 
   linkcolor    = blue, 
   citecolor   = red 
}
\usepackage[usenames,dvipsnames]{color}

\input{xy}\xyoption{all}

\newcommand{\jonly}[1]{}
\newcommand{\aronly}[1]{#1}

\def\Z{{\mathbb Z}} \def\R{{\mathbb R}}  
\long\def\comment#1\endcomment{}
\def\Hom{\mathop{\fam0 Hom}}
\def\pr{\mathop{\fam0 pr}}
\def\diag{\delta}
\def\diam{\mathop{\fam0 diam}}
\def\pii{\rho}
\def\id{\mathop{\fam0 id}}
\def\st{\mathop{\fam0 st}}

\def\Int{\mathop{\fam0 Int}}
\def\rel{\mathop{\fam0 rel}}
\def\im{\mathop{\fam0 im}}

\renewcommand{\t}[1]{\ensuremath{\widetilde{#1}}}

\theoremstyle{plain}
\newtheorem{Theorem}{Theorem}[section]
\newtheorem{Lemma}[Theorem]{Lemma}

\newtheorem{Proposition}[Theorem]{Proposition}

\theoremstyle{definition}
\newtheorem{Remark}[Theorem]{Remark}

\begin{document}

\title{Eliminating higher-multiplicity intersections in the metastable dimension range}

\author{Arkadiy Skopenkov\footnote{
Email: \texttt{skopenko@mccme.ru}. \texttt{https://users.mccme.ru/skopenko}.
\aronly{Independent University of Moscow, Moscow Institute of Physics and Technology.
Research supported by Simons-IUM Fellowship and by the D. Zimin's Dynasty Foundation Grant.}
\newline
I am grateful to S. Avvakumov, R. Karasev, S. Melikhov and an anonymous referee for helpful discussions of this paper, and I. Mabillard and U. Wagner for discussions of \cite{MW16}.}
}

\date{}
\maketitle

\begin{abstract}
The procedure to remove {\it double} intersections called the {\it Whitney trick} is one of the main tools in the topology of manifolds.
The analogues of Whitney trick for {\it $r$-tuple} intersections were `in the air' since 1960s.
However, only recently they were stated, proved and applied to obtain interesting results.
Here we prove and apply the $r$-fold Whitney trick when general position $r$-tuple intersection has positive dimension.
A continuous map $f\colon M \to B^d$ from a manifold with boundary to the $d$-dimensional ball is called {\it proper}, if $f^{-1}(\partial B^d)=\partial M$.

{\bf Theorem.} {\it Let $D=D_1\sqcup\ldots\sqcup D_r$ be disjoint union of $k$-dimensional disks,
and $f:D\to B^d$ a proper map such that $f\partial D_1\cap\ldots\cap f\partial D_r=\emptyset$ and
the map
$$f^r:\partial(D_1\times\ldots\times D_r)\to (B^d)^r-\{(x,x,\ldots,x)\in(B^d)^r\ :\ x\in B^d\}$$
extends continuously to $D_1\times\ldots\times D_r$.
If $rd\ge (r+1)k+3$, then there is a proper map $\overline f:D\to B^d$ such that
$\overline f=f$ on $\partial D$ and $\overline fD_1\cap\ldots\cap \overline fD_r=\emptyset$. }
\end{abstract}



\tableofcontents

\comment

Dear Editors:

Concerning the article
Eliminating higher-multiplicity intersections in the metastable dimension range
submitted on ??? 2022 I feel obliged to inform you that there are

* a negative review to the previous version of the paper (close to the submitted version);

* a detailed, public, and not publicly opposed justification that the review is incompetent, see https://arxiv.org/abs/2101.03745v3, Example 6.2 (and Example 6.1 for a wider overview).

(The review is quoted by parts in Example 6.2 and is available in its complete form upon request.)

So let me suggest, if I may, that a report on the submitted paper from a specialist in topological combinatorics could be read critically, and balanced by a report from a generic topologist.
(Because the reason for writing the incompetent report and for its approval by the Editors of Israel J. Math. is likely to be an attempt to suppress my detailed, public, and not publicly opposed criticism of some papers in topological combinatorics.)

Thank you,
Arkadiy Skopenkov

{\bf Hypergraph drawings without multiple self-intersections. A. Skopenkov}

Algorithms for recognizing graph planarity are well-known.
The classical topics in mathematics and computer science is recognizing realizability of hypergraphs in $d$-dimensional Euclidean space $\mathbb R^d$.
In these studies, as well as in studies of Radon-Tverberg-type problems from topological combinatorics, the following notion appeared, see e.g. survey \cite{Sk18}.
Denote by $\Delta_n$ a simplex with $n$ vertices.
The {\bf body} (or geometric realization) $|K|$ of a hypergraph $K=(V,E)$ is the union of faces of $\Delta_{|V|}$ corresponding to elements of $E$.
An {\bf almost $r$-embedding} of a hypergraph $K$ is a piecewise-linear (PL) map $f\colon |K|\to \mathbb R^d$ such that the images of any $r$ pairwise disjoint faces of $K$ do not have a common point.
A {\bf $k$-dimensional hypergraph} is a hypergraph whose maximal hyperedge has $k$ elements.
By general position, for $(r-1)d>rk$ every $k$-dimensional hypergraph almost $r$-embeds in $\mathbb R^d$.
For fixed $r$ and $d$, a polynomial algorithm for recognition almost $r$-embeddability of $k$-hypergraphs in $\mathbb R^d$ was constructed

$\bullet$ in \cite{MW15, AMSW} for $(r-1)d=rk$;

$\bullet$ in \cite{Sk17} for $rd\ge(r+1)k+3$, see also \cite{MW16}.

\bibitem[AMS+]{AMSW} \emph{S. Avvakumov, I. Mabillard, A. Skopenkov and U. Wagner.}
Eliminating Higher-Multiplicity Intersections, III. Codimension 2, Israel J. Math., to appear, arxiv:1511.03501.

\bibitem[MW15]{MW15} \emph{I. Mabillard and U. Wagner.}
Eliminating Higher-Multiplicity Intersections, I. A Whitney Trick for Tverberg-Type Problems. arXiv:1508.02349.

\bibitem[MW16]{MW16} \emph{I. Mabillard and U. Wagner.} Eliminating Higher-Multiplicity Intersections, II. The Deleted Product Criterion in the $r$-Metastable Range. arxiv:1601.00876.

\bibitem[Sk17]{Sk17} \emph{A. Skopenkov.}
Eliminating higher-multiplicity intersections in the metastable dimension range,  arxiv:1704.00143.

\bibitem[Sk18]{Sk18} \emph{A. Skopenkov.} Invariants of graph drawings in the plane.
Abridged version: Arnold Math. J., 2020; full version: arXiv:1805.10237.

{\bf Eliminating higher-multiplicity intersections. A. Skopenkov}

\bigskip
{\bf RESEARCH REPORT}

\smallskip
We study conditions under which a finite simplicial complex $K$ can be mapped to $\R^d$ without $r$-fold
intersections.
An \emph{almost $r$-embedding} is a map $f\colon K\to \R^d$ such that the images of any $r$ pairwise disjoint simplices of $K$ do not have a common point.

\begin{Theorem}\label{t:tve} \cite{AMSW} If $r$ is not a prime power and $d\geq 2r+1$, then there is an almost $r$-embedding of the $(d+1)(r-1)$-dimensional simplex in $\R^d$.
\end{Theorem}

This is a counterexample to the topological Tverberg conjecture improving earlier counterexamples (for $d\geq 3r$) based on a series of papers by M.\ \"Ozaydin, M.\ Gromov, P.\ Blagojevi\'{c}, F.\ Frick, G.\ Ziegler, I. Mabillard and U. Wagner.
Theorem \ref{t:tve} is obtained by proving the case $d=2r=k+2$ of Theorem \ref{t:mmw} below.
As another application, we classify \emph{ornaments} $f\colon S^3 \sqcup S^3 \sqcup S^3\to \R^5$ up to \emph{ornament concordance}.


\begin{Theorem}[Local Disjunction]\label{l:ldm} Assume that $D=D_1\sqcup\ldots\sqcup D_r$ is disjoint union of $k$-dimensional disks,
$$(*)\quad\text{either}\quad rd\ge(r+1)k+3\quad\text{or}\quad d=2r=k+2,$$
and $f:D\to B^d$ a proper PL (smooth) map such that $f\partial D_1\cap\ldots\cap f\partial D_r=\emptyset$.
If the map
$$f^r:\partial(D_1\times\ldots\times D_r)\to (B^d)^r-\{(x,x,\ldots,x)\in(B^d)^r\ |\ x\in B^d\}$$
extends to $D_1\times\ldots\times D_r$, then there is a proper PL (smooth) map $\overline f:D\to B^d$ such that
$\overline f=f$ on $\partial D$ and $\overline fD_1\cap\ldots\cap \overline fD_r=\emptyset$.
\end{Theorem}

\begin{Theorem}[Mabillard-Wagner]\label{t:mmw}
Let $K$ be a finite $k$-dimensional complex, (*) holds and $(d,k,r)\ne(4,2,2)$.
There exists an almost $r$-embedding $f:K\to\R^d$ if and only if there exists a $\Sigma_r$-equivariant map
$$K^{\times r}_{\Delta}:= \bigcup \{ \sigma_1\times\cdots\times\sigma_r
\ : \sigma_i \textrm{ a simplex of }K,\ \sigma_i \cap \sigma_j = \emptyset \mbox{ for every $i \neq j$} \}
\to \R^{d\times r}-\diag\phantom{}_r.$$
\end{Theorem}

Theorems \ref{l:ldm} and \ref{t:mmw} were proved (or essentially proved)

$\bullet$ for $d=sr=k+s$, where $s\ge3$, in \cite{MW15},

$\bullet$ for $d=2r=k+2$ in \cite{AMSW},

$\bullet$ for $rd\ge(r+1)k+3$ \cite{MW16, Sk17}.

Theorem \ref{t:mmw} implies the existence of a polynomial algorithm for checking almost $r$-embeddability
(for fixed $k,d,r$ such that (*) holds and $(d,k,r)\ne(4,2,2)$), cf. \cite[Corollary 5]{MW15}.
Theorem \ref{t:mmw} also reduces improvement of counterexamples to the topological Tverberg conjecture to an algebraic problem, cf. \cite[paragraph before Theorem 1.6]{Sk16}.

It follows from work of M.\ Freedman, V.\ Krushkal, and P.\ Teichner that the analogue of Theorem \ref{t:mmw} for $d=2r=k+2=4$ is false.
We prove a lemma on singular higher-dimensional Borromean rings, yielding an elementary
proof of the counterexample.
This lemma, together with the technique of \cite{MTW08}, is one of the main ingredients for the following result.

\begin{Theorem} \cite{ST17} Fix integers $d,k\ge2$ such that $d=\frac{3k}2+1$.

(a) Assume that P~$\ne$~NP.
Then there exists a finite $k$-dimensional complex $K$ that does not admit an
almost 2-embedding in $\R^d$ but for which there exists an equivariant map
$K^{\times 2}_{\Delta}\to\R^{d\times 2}-\diag\phantom{}_2$.

(b) The algorithmic problem of recognition of almost 2-embeddability of
finite $k$-dimensional complexes in $\R^d$ is NP-hard.
\end{Theorem}

\endcomment

\section{Introduction and main results}\label{s:mr}

In this paper we prove and apply a result on removal $r$-tuple intersection ($r$-fold Whitney trick)
in the dimension range when general position $r$-tuple intersection has positive dimension.

We omit `continuous' for maps and extensions, and `compact smooth, possibly with boundary' for manifolds.
Let $B^d:=\{x\in\R^d\ :\ x_1^2+\ldots+x_d^2\le1\}$ and $S^{d-1}=\partial B^d$ be the standard ball and sphere.
Denote by $\partial M$ and $\Int M$ the boundary and the interior of a manifold $M$, respectively.
A map $f\colon M \to B^d$ from a manifold is called {\bf proper}, if $f^{-1}S^{d-1}=\partial M$
(this should not be confused with the term `proper' from general topology).
Remarks of this text are formally not used in the sequel (except for (\ref{r:motiv}.*) and the notation $n$ of Remark \ref{r:motiv}.b).

\begin{Remark}[motivation]\label{r:motiv}
(a) The procedure to remove {\it double} intersections of a map called the {\it Whitney trick}
is one of the main tools in the topology of manifolds.
The first application was the 1944 Strong Whitney Embedding Theorem:
{\it any compact $k$-dimensional manifold embeds into $\R^{2k}$.}

Consider a proper general position map $f:N_1\sqcup N_2\to B^d$ of a disjoint union of $n_1$- and $n_2$-dimensional manifolds.
The Whitney trick was first proved (and applied) for $d=n_1+n_2$, i.e., for the case when general position intersection $fN_1\cap fN_2$ is a finite number of points.
The condition for removing intersection was stated in terms of the algebraic intersection number.
The following version of the Whitney trick is a predecessor of Theorems \ref{t:hawe} and \ref{l:ldm}.

{\it If $d-3\ge n_1,n_2$ and the algebraic
intersection number $fN_1\cdot fN_2\in\Z$ is zero, then there is a proper general position map}
$$\overline f:N_1\sqcup N_2\to B^d\quad\text{such that}\quad \overline f=f\text{ on }
\partial N_1\sqcup N_2\quad\text{and}\quad \overline fN_1\cap \overline fN_2=\emptyset.$$
See a classical exposition e.g. in \cite[Whitney Lemma 5.12]{RS72}, cf. \cite{Skw}.

The generalization to the (`metastable') case $2d\ge2n_1+n_2+3$ is essentially due to Haefliger and Weber \cite{Ha63, We67}.
The condition for removing intersection was stated in terms of extension of certain map (of `configuration space'), see Theorem \ref{t:hawe} below.
This is related to `h-principle for embeddings' \cite[2.1.E]{Gr86}.
Guessing the statement, i.e., passing from algebraic sum of the intersection points to the extension condition is easy.
Proof of the necessity is easy.
Proof of the sufficiency (namely, removal of intersections) is hard, see exposition e.g. in \cite[\S8]{Sk06}.

(b) The analogues of Whitney trick for removing {\it $r$-tuple} intersections
were `in the air' since 1960s, see \cite[Remark 3.6]{Sk16}.
However, only recently they were stated, proved and applied to obtain interesting results
\cite{MW15, AMSW, MW16}, \cite[proof of Theorem 1.1 in pp. 7-8]{Me17}, \cite{Me18}.
Consider a proper map $f:N_1\sqcup\ldots\sqcup N_r\to B^d$ of a disjoint union of $n_1,\ldots,n_r$-dimensional manifolds, and let
$$n=n_1+n_2+\ldots+n_r.$$
The analogues were proved for the case when general position intersection $fN_1\cap\ldots\cap fN_r$
is a finite number of points, i.e., for
$$(r-1)d=n.$$
Here we prove and apply the $r$-fold Whitney trick when general position $r$-tuple intersection has positive dimension. More precisely, we consider the (`metastable') case
$$(\ref{r:motiv}.*)\quad rd\ge n+n_i+3\quad\text{for each }i=1,\ldots,r.$$
See the Metastable Local Disjunction Theorem \ref{l:ldm} below.
See \cite{Sk16} for a survey and Remark \ref{r:imm} for relations to the above papers.

(c) Maps without $r$-fold intersections (more precisely, $r$-almost embeddings defined before Theorem \ref{t:mmw} below) are studied in topological combinatorics even more actively than in topology.
The main motivation for $r$-fold Whitney trick in \cite{MW15, AMSW, MW16} were counterexamples
to the topological Tverberg conjecture, which was considered a central problem of topological combinatorics.
See \cite{Sk16} for a survey (and for references to other surveys).
The counterexamples are maps of $k$-complexes to $\R^d$ without certain $r$-tuple self-intersections.
The counterexamples were obtained for the case when general position $r$-tuple self-intersection is a finite number  of points, i.e., for $(r-1)d=rk$.
One of the main results (Metastable Mabillard-Wagner Theorem \ref{t:mmw}) allows to obtain in \cite{AKS} improved counterexamples to the topological Tverberg conjecture.
These counterexamples appear for the case when general position $r$-tuple self-intersection has positive dimension, i.e., for $rd\ge(r+1)k+3$.
\end{Remark}

Let $B^{d\times r}:=(B^d)^r$.
Let $\R^{d\times r}:=(\R^d)^r$ be the set of real $d\times r$-matrices.
Denote
$$
\diag\phantom{}_r:=\{(x,x,\ldots,x)\in \R^{d\times r}\ |\ x\in\R^d\}\quad\text{and}\quad N^\times:=N_1\times\ldots\times N_r.$$
Define a map
$$f^r:N^\times\to B^{d\times r}\quad\text{by}\quad f^r(x_1,\ldots,x_r):=(fx_1,\ldots,fx_r).$$
We use the same notation $f^r$ for maps defined by the same formula on subsets of $N^\times$,
and assuming values in $B^{d\times r}-\diag_r$; the domain and the range are specified, so no confusion will appear.

The following theorem is a (`relative link map') version of the celebrated Weber Theorem 
\cite[the Weber Theorem 8.1]{Sk06}, and is (a stronger version of) the case $r=2$ 
of the Metastable Local Disjunction Theorem \ref{l:ldm} below.

\begin{Theorem}\label{t:hawe}
Let $f:N_1\sqcup N_2\to B^d$ be a proper map of a disjoint union of $n_1$- and $n_2$-dimensional manifolds such that $2d\ge2n_2+n_1+3$ and $f\partial N_1\cap f\partial N_2=\emptyset$.
There exists a map $\overline f:N\to B^d$ such that
$$\overline f=f \quad\text{on}\quad \partial N_1\sqcup N_2\quad\text{and}\quad
\overline fN_1\cap \overline fN_2=\emptyset$$
if and only if the map\footnote{We have $f^2\partial(N_1\times N_2)\cap\diag_2=\emptyset$ because $f\partial N_1\cap f\partial N_2=\emptyset$ and $f$ is proper.}
$f^2:\partial(N_1\times N_2)\to B^{d\times 2}-\diag_2$ extends to $N_1\times N_2$.
\end{Theorem}

Proof of the `only if' part is simple and is analogous to the proof of the `only if' part of Theorem \ref{l:ldm} below.
Concerning the `if' part  see Remark \ref{r:ifhawe} below.

An {\bf $(n_1,\ldots,n_r)$-Whitney map} is a proper map $f:N\to B^d$ of disjoint union
$N=N_1\sqcup\ldots\sqcup N_r$ of manifolds\footnote{The most important particular case sufficient for Theorem \ref{t:mmw} below is when each $N_i$ is a disk.
For the inductive proof of Theorem \ref{l:ldm} in this particular case we need the above more general notion.}
of dimensions $n_1,\ldots,n_r$ such that the equality (\ref{r:motiv}.*) holds.
Denote
$$w:=n-(r-1)d.$$
This is the dimension of general position $r$-tuple intersection $fN_1\cap\ldots\cap fN_r$.
The inequality (\ref{r:motiv}.*) is equivalent to $d\ge n_i+w+3$ (which is more convenient for the proof of Theorem \ref{l:ldm}).

A manifold is {\bf stably parallelizable} if it allows a codimension one immersion with a non-zero normal vector field.
A manifold $M$ is {\bf $k$-connected} if for every $j=0,\ldots,k$ any map $S^j\to M$ is homotopic to the constant map.

\begin{Theorem}[Metastable Local Disjunction]\label{l:ldm}
Let $f:N\to B^d$ be an $(n_1,\ldots,n_r)$-Whitney map of a disjoint union of $(w+1)$-connected stably  parallelizable manifolds such that $f\partial N_1\cap\ldots\cap f\partial N_r=\emptyset$.
There exists a map $\overline f:N\to B^d$ such that
$$\overline f=f \quad\text{on}\quad \partial N \quad\text{and}\quad
\overline fN_1\cap\ldots\cap \overline fN_r=\emptyset$$
if and only if the map\footnote{We have $f^r\partial N^\times\cap \diag_r=\emptyset$ because $f\partial N_1\cap\ldots\cap f\partial N_r=\emptyset$ and $f$ is proper.}
$f^r:\partial N^\times\to B^{d\times r}-\diag_r$ extends to $N^\times$.
\end{Theorem}

\begin{proof}[Proof of the `only if' part]
Take the linear homotopy $f_t$ betwen $f_0=f$ and $f_1=\overline f$.
We have $f_t=f_0=f_1$ on $\partial N$.
Hence $f_t^r:\partial N^\times\to B^{d\times r}-\diag_r$ is a homotopy from
$f^r:\partial N^\times\to B^{d\times r}-\diag_r$
to the restriction to $\partial N^\times$ of $\overline f^r:N^\times\to B^{d\times r}-\diag_r$.
So by the Borsuk Homotopy Extension Theorem\footnote{\label{f:bhet}This theorem states that if $(K,L)$ is a polyhedral pair, $Z\subset\R^m$, \ $F:L\times I\to Z$ is a homotopy, and $g:K\to Z$ is a map such that $g|_L=F|_{L\times0}$, then $F$ extends to a homotopy $G:K\times I\to Z$ such that $g=G|_{K\times0}$.}
 $f^r:\partial N^\times\to B^{d\times r}-\diag_r$ extends to $N^\times$.
\end{proof}

Note that the `if' part of Theorems \ref{t:hawe} and \ref{l:ldm}
is trivial when $\partial N_i=\emptyset$ for some $i$.

A {\bf complex} is a collection of closed faces (=simplices) of some simplex.\footnote{This is an abbreviation of `an (abstract) finite simplicial complex'.
In combinatorial terms, a complex is a collection of subsets of a finite set such that
if a subset $A$ is in the collection, then each subset of $A$ is in the collection,
A close but different notion widely studied in combinatorics is {\it hypergraph}.}
A {\it $k$-complex} is a complex containing at most $k$-dimensional simplices.
The {\it body} (or geometric realization) $|K|$ of a complex $K$ is the union of simplices of $K$.
Thus continuous
maps $|K|\to\R^d$ and
$|K|\to S^m$ are defined.
Below we abbreviate $|K|$ to $K$; no confusion should arise.

A map $f\colon K\to\R^d$ of a complex is an {\bf almost $r$-embedding} if
$f\sigma_1\cap\ldots\cap f\sigma_r=\emptyset$
whenever $\sigma_1,\ldots,\sigma_r$ are pairwise disjoint simplices of $K$.
Motivations for study this notion are presented in Remark \ref{r:motiv}.c and in \cite[\S1]{Sk16}.

Denote by $\Sigma_r$ the permutation group of $r$ elements.
The group $\Sigma_r$ acts on $\R^{d\times r}$ by permuting the columns.
For a complex $K$ let
$$K^{\times r}_{\Delta}:= \bigcup \{ \sigma_1 \times \cdots \times \sigma_r
\ : \sigma_i \textrm{ a simplex of }K,\ \sigma_i \cap \sigma_j = \emptyset \mbox{ for every $i \neq j$} \}.$$
The group $\Sigma_r$ acts the set $K^{\times r}_{\Delta}$ by permuting the points in every $r$-tuple
$(p_1,\ldots, p_r)$.

\begin{Theorem}[Metastable Mabillard-Wagner Theorem]\label{t:mmw}
Assume that $rd\ge(r+1)k+3$ and $K$ is a $k$-complex.
There exists an almost $r$-embedding $f:K\to\R^d$ if and only if there exists a $\Sigma_r$-equivariant map
$K^{\times r}_{\Delta}\to \R^{d\times r}-\diag\phantom{}_r$.
\end{Theorem}


The `only if' part of Theorem \ref{t:mmw} is clear: take $f^r$ as equivariant map.
The `if' part is implied by Theorem \ref{t:mmwi} below (for $E_0=\emptyset$
and $E_1=K^{\times r}_{\Delta}$).
Theorem \ref{t:mmwi} (and thus Theorem \ref{t:mmw}) is a `globalization' of the two local modifications:  Metastable Local Disjunction Theorem \ref{l:ldm} and the Metastable Local Realization Theorem \ref{l:ldmr}.
The latter is an analogue of `van Kampen finger moves' \cite[figures 4.2.V and 8.3, \S8]{Sk06},
and is `the second part' of the generalized Whitney trick.

For (conjectural) improvements of our results see \aronly{Remark \ref{r:conj}.}\jonly{\cite[Remark 4.3]{Sk17}.}

\begin{Remark}\label{r:sphere} (a) Theorem \ref{t:mmw} and the unpublished result \cite[Theorem 4]{MW16} of \cite{FV21} imply \cite[Corollary 5]{MW16}, \cite[Theorem 1.6]{FV21}: {\it for fixed $k,d,r$ such that $rd\ge(r+1)k+3$ there exists a polynomial algorithm for checking almost $r$-embeddability of complexes in $\R^d$.}

(b) Denote by $S^{d(r-1)-1}=S^{d(r-1)-1}_{\Sigma_r}$ the set formed by all those of real $d\times r$-matrices, for which the sum in each row zero, and the sum of squares of the matrix elements is 1.
This set is homeomorphic to the sphere of dimension $d(r-1)-1$.

For $x_1,\ldots,x_r\in\R^d$ which are not all equal define
$$
S:=x_1+\ldots+x_r,\quad \pii':=\left(x_1-\frac Sr,\ldots,x_r-\frac Sr\right)
\quad\text{and}\quad \pii:=\frac{\pii'}{|\pii'|}.
$$
This defines a map $\pii=\pii_r:\R^{d\times r}-\diag\phantom{}_r\to S^{d(r-1)-1}_{\Sigma_r}.$
Recall that this map is a $\Sigma_r$-equivariant homotopy equivalence.
Analogously $B^{d\times r}-\diag_r$ is
$\Sigma_r$-equivariantly homotopy equivalent to $S^{d(r-1)-1}_{\Sigma_r}$.
\end{Remark}

\begin{Remark}[our methods and their novelty]\label{r:imm} (a)
Passage from the case $(r-1)d=n$ of the papers cited below in (a) to the case $rd\ge n+n_i+3$ considered in this paper is non-trivial because here general position $r$-tuple intersections are no longer isolated points.
This makes surgery of intersection more complicated, cf. (b).
More importantly, this brings in `extendability of $f^r$' obstruction, which is harder to work with than the
`sum of the signs of the global $r$-fold points' integer obstruction.
Cf. \cite[Remark 3.1.b]{AMSW}.

More precisely, proofs in this paper

$\bullet$ generalize proofs from \cite[\S2.1]{AMSW}, \cite[\S3.5]{Sk16},
\cite[proof of Theorem 1.1 in pp. 7-8]{Me17}, \cite{Me18} for $r=3$;

$\bullet$ do not generalize those proofs for $r\ge4$, and do not generalize the proof from \cite{MW15}
for any $r$ (but still are closely related to all those proofs).

(b) The main achievement is invention of those formulations of Theorem \ref{l:ldm} and Proposition \ref{l:ldmin} that allow proofs by induction.
This required highly-connectedness and stable parallelizability in Theorem \ref{l:ldm}, and a subtle transversality property in Proposition \ref{l:ldmin} (i.e., in the definition of a frimmersion in \S\ref{s:meta}).
For Theorem \ref{t:mmw} we also need a realization theorem (Theorem \ref{l:ldmr}).

We use induction on $r$.
The inductive step is, roughly speaking, passing from $fN_1,\ldots,fN_r$ in $B^d$ to $fN_1,\ldots,fN_{r-2},fN_{r-1}\cap fN_r$ in $B^d$.
We need  $fN_{r-1}\cap fN_r$ to be a manifold.
For this we use differential topology.
More precisely, we use Smale-Hirsch immersion theory, in particular, density principle.
This allows to make $f$ a framed immersion.
Thus instead of Theorem \ref{l:ldm} we prove by induction its version (Proposition \ref{l:ldmin}) for 
frimmersions, i.e., for framed immersions with certain subtle transversality property.
The latter is required because the self-transversality of $fN_1,\ldots,fN_r$ in $B^d$ is lost by passing to $fN_1,\ldots,fN_{r-2},fN_{r-1}\cap fN_r$. 
Indeed, let us present an example for $r=2$. 
Take the natural map of the circle to the 8-figure in the plane. 
Multiply this map with the interval to obtain a self-transverse map $N_1=S^1\times[0,1]\to\R^3$. 
Take a map $N_2=D^2\to\R^3$ `orthogonal' to the interval.
Then $fN_1\cap fN_2$ is the 8-figure, which is not self-transversal in $\R^3$. 
 
In order to check that the extendability assumption is preserved by passing to $fN_1,\ldots,fN_{r-2},$ $fN_{r-1}\cap fN_r$, we use Pontryagin-Thom correspondence between homotopy classes and framed bordism classes.
For this we need high enough connectedness of $N_1,\ldots,N_r$.
To make $fN_{r-1}\cap fN_r$ high enough connected, we use the Surgery of Intersection Lemma \ref{l:surg} for frimmersions.


(c) Theorems \ref{l:ldm} and \ref{l:ldmr} generalize a relaxed version of \cite[Lemma 10]{MW16}\aronly{, cf. Remark \ref{r:conj}.e}.
Theorem \ref{t:mmw} is \cite[Theorem 2]{MW16}.
In \aronly{\S\ref{s:apmawa} (cf. \cite{Sk17o})}\jonly{\cite[\S6]{Sk17}} it is shown that the paper \cite{MW16} does not provide a reliable proof of  \cite[Lemma 10 and Theorem 2]{MW16}.
The revision of \cite{MW16} planned by U. Wagner in his response to the criticism (see
\aronly{Remark \ref{r:uw}}\jonly{\cite[Remark 6.2]{Sk17}}) is not publicly available five years after that plan.
{\it In spite of all that I call Theorem \ref{t:mmw} Metastable Mabillard-Wagner Theorem, in order to concentrate on mathematics and on different reliability standards in current mathematical research \cite{Sk21d}, not on priority questions.}

Proofs in this paper and in \cite{MW16} are similar because they use and extend known methods; the new parts of proofs are essential and are different.

For Theorems \ref{l:ldm} and \ref{l:ldmr} we use the Surgery of Intersection Lemma \ref{l:surg}, see the references after its statement.\footnote{These references extend the surgery of a manifold \cite{Mi61} to a surgery of double intersection of a map from a manifold used in \cite[\S4.1]{MW16}.
In \cite[\S4.1]{MW16} this surgery of double intersection was used with a reference to the `predecessor' paper \cite{Mi61}, but without the above references; an interesting but not complete \aronly{(\S\ref{s:apmawa})}\jonly{\cite[\S6]{Sk17}} attempt to prove a PL analogue was made.
Lack of the above references in \cite{MW16} may result in exaggerating the overlap of methods of \cite{MW16} and of this paper.}
For Theorem \ref{t:mmwi} we use `engulfing to a ball' (Lemma \ref{l:ball}) together with `reduction to Local Disjunction and Realization Theorems' generalizing \cite[\S5]{We67}.\aronly{\footnote{\label{f:disj} For generalizations in other directions see \cite[Proposition 2.3]{Sk00}, \cite[the Disjunction Theorem 3.1]{Sk02}; for a survey exposition see \cite[\S8.4]{Sk06}.}}

In this paper the proofs of Theorems \ref{l:ldm} and \ref{l:ldmr} are shorter (than unreliable proof of \cite[Lemma 10]{MW16}) because they are deduced by approximation from their analogues for smooth framed immersions, see Propositions \ref{l:ldmin} and \ref{l:ldmrin} below.
This allows to use vector bundles (instead of block bundles as in \cite{MW16}) and Pontryagin-Thom correspondence.
The `desuspension' part of the argument is simpler than in \cite{MW16} because the inductive step is, roughly speaking, passing from $fN_1,\ldots,fN_r$ in $B^d$ to $fN_1,\ldots,fN_{r-2},fN_{r-1}\cap fN_r$ in $B^d$ (not to $fN_1\cap fN_r,\ldots,fN_{r-1}\cap fN_r$ in $fN_r$ as in \cite{MW16}).
Proof of Theorem \ref{t:mmw} is shorter than \cite[\S3]{MW16} because I state Theorem \ref{t:mmwi} as a simple
way to organize the induction, cf. \cite[Disjunction Theorem 3.1]{Sk02}.
\end{Remark}

\section{Proof of the Metastable Local Disjunction Theorem}\label{s:meta}

\begin{Remark}\label{r:ifhawe}
The `if' part of Theorem \ref{t:hawe} for a PL map $f$ is \cite[Theorem 1.3]{Sk00} for $N=\emptyset$ and manifolds $Q=N_1$, $P=N_2$.
For $m>p+q$ ($d>n_1+n_2$) that theorem holds by general position, while for $m\le p+q$ ($d\le n_1+n_2$) we have
$m\ge q+(p+q-m)+3\ge q+3$.
Thus we need the easier case $0\le m-3\ge p,q$ of that theorem essentially proved in \cite{We67}.
For this case, the assumption $p\le q$ ($n_2\le n_1$) is not used in that theorem.
(Instead of using \cite[Theorem 1.3]{Sk00} for $p>q$ we can have the assumptions $n_2\le n_1$ and
$n_r\le\ldots\le n_1$ in Theorems \ref{t:hawe}, \ref{l:ldmr2} and Propositions \ref{l:ldmin}, \ref{l:ldmrin}, respectively.)
The `if' part of Theorem \ref{t:hawe} for the general case is obtained from the PL case by taking a small proper homotopy from $f$ to a PL map, analogously to the proof of Theorem \ref{l:ldm} in \S\ref{s:meta}.

\aronly{By \cite[Theorem 1.3]{Sk00} for $N=\emptyset$ and manifolds $Q=N_1$, $P=N_2$, in Theorem \ref{t:hawe} we can replace `$2d\ge2n_2+n_1+3$ ... $\overline f=f$ on $\partial N_1\sqcup N_2$' to `$2d\ge2n_2+n_1+2$ and $n_2\le n_1\le d-2\ge1$ ... $\overline f=f$ on $\partial N_1\sqcup\partial N_2$'.
Cf. \cite[2nd paragraph in p. 3]{Sk02}.

Alternatively, Theorem \ref{t:hawe} can perhaps be proved by induction on $d$, the inductive step reducing the statement for $(n_1,n_2,d)$ to the statement for $(n_1,n_2-1,d-1)$, cf. \cite[\S5]{RS72}.}
\end{Remark}

For a map $g:X\to B^d$ the {\bf self-intersection set} of $g$ is
$$S(g):=\{x\in X\ :\ |g^{-1}gx|>1\}.$$
{\it We shorten `smooth immersion' to `immersion'.}
An immersion is {\bf framed} if it admits a full rank collection of normal vector fields.
An {\bf $s$-frimmersion} is a proper framed immersion $f:N_1\sqcup\ldots\sqcup N_r\to B^d$ of disjoint union of   manifolds such that


$\bullet$ $f|_{N_1},\ldots,f|_{N_r}$ are transverse to each other,

$\bullet$ $f|_{N_i}$ is self-transverse for every $i=1,\ldots,r-1$, and

$\bullet$ $S(f|_{N_r})$ is the image of a self-transverse immersion to $N_r$ of a finite disjoint union of manifolds of dimensions at most $s$.

\begin{Proposition}\label{l:ldmin}
Let $f:N\to B^d$ be an $(n_1,\ldots,n_r)$-Whitney
$(n_r-w-3)$-frimmersion
such that

(\ref{l:ldmin}.1) $f\partial N_1\cap\ldots\cap f\partial N_r=\emptyset$,

(\ref{l:ldmin}.2) the map $f^r:\partial N^\times\to B^{d\times r}-\diag_r$ extends to $N^\times$, and

(\ref{l:ldmin}.3) if $r\ge3$, then every $N_i$ is $(w+1)$-connected.

Then there exists a map
$\overline f:N\to B^d$ such that
$$\overline f=f \quad\text{on}\quad N_r\cup\partial N \quad\text{and}\quad
\overline fN_1\cap\ldots\cap \overline fN_r=\emptyset.$$
\end{Proposition}

\begin{Lemma}\label{l:hprin}
Assume that $V$ is a manifold, $f,h:V\to B^d$ are proper maps,
$h$ is an immersion, and $\dim V<d$.
Then for each $\varepsilon>0$ there is a proper homotopy $F:V\times I\to B^d$
between $f$ and an immersion regular homotopic to $g$ such that
$\diam F(v\times I)<\varepsilon$ for every $v\in V$.
\end{Lemma}

This is the well-known Smale-Hirsch density principle \cite[Theorems 5.7 and 5.10]{Hi59}, \cite[\S1.2.2.A]{Gr86}
(because every two proper maps $N\to B^d$ are properly homotopic;
for an application in a similar situation see e.g. \cite[the paragraph before (4.1)]{Ko88}).
The definition of regular homotopy can be found e.g. in \cite[Definition 2.7]{Ko13}; we only need that being framed is preserved under regular homotopy.

Denote by $\pr_X$ the projection of a Cartesian product to the factor $X$.

\begin{proof}[Proof of the `if' part of Theorem \ref{l:ldm} assuming Proposition \ref{l:ldmin}]
Since every $N_i$ is stably parallelizable, by the Smale-Hirsch immersion theorem there is a proper framed immersion $g:N\to B^d$.
Then by Lemma \ref{l:hprin} there is a small homotopy $F:N\times I\to B^d$
from $F_0:=f$ to a proper immersion $F_1:N\to B^d$ regular homotopic to $g$.
Then $F_1$ is also framed.
By general position we may assume that $F_1$ is self-transverse.

Since the homotopy $F$ is small, we have $F_t^r(\partial N^\times)\cap\diag_r=\emptyset$ for each $t$.
Since $F_0^r=f^r:\partial N^\times\to B^{d\times r}-\diag_r$ extends to $N^\times$, by the Borsuk Homotopy Extension Theorem (see footnote \ref{f:bhet}) $F_1^r$ extends to $N^\times$.
Since $F_1$ is self-transverse, $S(F_1|_{N_r})$ is the image of a self-transverse immersion to $N_r$ of a finite disjoint union of manifolds
of dimensions at most $2n_r-d\le n_r-w-3$, which inequality follows from $d\ge n_r+w+3$.
Then applying Proposition \ref{l:ldmin} to $F_1$ we obtain a map $\overline{F_1}:N\to B^d$.
Then the required extension $\overline f$ is defined as `the union' of $F|_{\partial N\times I}\times\pr_I$ and $\overline{F_1}$:
$$\xymatrix{N  \ar[r]_(0.3){\cong}^(0.3){h_N} \ar[d]^{\overline f} &
\partial N\times I \bigcup\limits_{\partial N\times0=\partial N} N
\ar[d]^{F|_{\partial N\times I}\times\pr_I\cup \overline{F_1}}\\
B^d \ar[r]_(0.3){\cong}^(0.3){h_{B^d}}  & S^{d-1}\times I\bigcup\limits_{S^{d-1}\times0=\partial B^d} B^d }.$$
Here $h_N,h_{B^d}$ are natural homeomorphisms.\aronly{\footnote{Since $F$ is not fixed on the boundary,
$\overline f$ need not coincide with $f$ on $N_r$ (unlike in Proposition \ref{l:ldmin}).}}
\end{proof}

For a map $f:P\sqcup Q\to B^d$ denote
$$M(f):=\{(x,y)\in P\times Q\ :\ fx=fy\}.$$
(This is a `resolution' of $fP\cap fQ$.)

\begin{Lemma}[Surgery of Intersection]\label{l:surg}
Let
$P$ and $Q$ be $k$-connected $p$- and $q$-manifolds,
and $f:P\sqcup Q\to B^d$ is a
$(q-k-2)$-frimmersion.
If $d-k-2\ge p,q$ and $p+q-d\ge2k+2\ge2$,
then there is a
$(q-k-2)$-frimmersion $f_1:P\sqcup Q\to B^d$ such that $f_1=f$ on $\partial P\sqcup Q$, and
$M(f_1)$ is a $k$-connected $(p+q-d)$-manifold.
\end{Lemma}

{\it Remark.} This lemma is analogous to \cite[Theorem 4.5 and appendix A]{HK98} and \cite[Theorem 4.7 and appendix]{CRS}; cf. \cite{Ha63}, \cite[Lemma 4.2]{Ha84},
\cite[The Surgery of Intersection Lemma 2.1]{AMSW}, \cite[\S4.1]{MW16}, \cite[proof of Theorem 1.1 in p. 7]{Me17}, \cite{Me18}.
The difference with \cite[Theorem 4.5]{HK98} and \cite[Theorem 4.7]{CRS} is that $f$ is a map of a disconnected manifold.
More importantly, $f_1=f$ on $\partial P\sqcup Q$, not only on $\partial P\sqcup \partial Q$.
Hence $f|_Q$ cannot be shifted to general position, so we need to assume that $\dim S(f|_Q)$ is small enough, and we assume that $f$ is a $(q-k-2)$-frimmersion.
Less importantly, we replace the conclusion on the classifying map with the simpler equivalent condition on $M(f_1)$.
See detailed description of references in \cite[Remark 2.3.a]{AMSW}, and detailed proof in \S\ref{s:apprsuin}.



\smallskip
{\it Proof of Proposition \ref{l:ldmin}: organization.}
We may assume that $w\ge0$, otherwise Proposition \ref{l:ldmin} holds by general position.
(A reader may first read the proof for $r=3$ and $w=0$.)

The proof is by induction on $r$.
We have $d\ge n_i+w+3\ge n_i+3$.
So the base $r=2$ follows from Theorem \ref{t:hawe};
we do not need $f$ to be an $(n_2-w-3)$-frimmersion.
Let us prove the inductive step $r-1\to r\ge3$.

\smallskip
{\it Proof of Proposition \ref{l:ldmin}: surgery.}
We have
$$2w+3=(w+n_{r-2}+3)+(w-n_{r-2})\le d+n_r+n_{r-1}+(r-3)d-(r-1)d=n_r+n_{r-1}-d,$$
where the inequality holds because $n_i<d$, and so is strict for $r\ge4$.
Hence we can apply the Surgery of Intersection Lemma \ref{l:surg} to $P=N_{r-1}$, $Q=N_r$, $k=w+1$ for $r=4$, and $k=w$ for $r=3$.
We obtain a map $f_1:N\to B^d$ satisfying the hypothesis of Proposition \ref{l:ldmin}, for which $f=f_1$ on
$N-\Int N_{r-1}$, and $M(f_1|_{N_{r-1}\sqcup N_r})$ is $(w+1)$-connected for $r\ge4$ and is $w$-connected for $r=3$.

Take a homotopy $f_t$ between $f=f_0$ and $f_1$.
Then $f_t^r$ is a homotopy between maps $f^r,f_1^r:\partial N^\times\to B^{d\times r}-\diag_r$.
So by the Borsuk Homotopy Extension Theorem (see footnote \ref{f:bhet}) the map
$f_1^r:\partial N^\times\to B^{d\times r}-\diag_r$ extends to $N^\times$.

Thus we can rename $f:=f_1$ and assume additionally that $M:=M(f|_{N_{r-1}\sqcup N_r})$ is $(w+1)$-connected for $r\ge4$ and is $w$-connected for $r=3$.

\smallskip
{\it Proof of Proposition \ref{l:ldmin}: construction of $g$ and checking some its properties.}
Denote
$$N_-:=N_1\sqcup\ldots\sqcup N_{r-2}.$$
Define $g:N_-\sqcup M\to B^d$ to be $f$ on $N_-$, and $f\pr_{N_r}$ on $M$.
If $M=\emptyset$, then the inductive step is clear.
So assume that $M\ne\emptyset$.

Let us check that $g$ is an $(n_1,\ldots,n_{r-2},n_r+n_{r-1}-d)$-Whitney map.
By transversality $\dim M=n_r+n_{r-1}-d$.
Hence $n_1+\ldots +n_{r-2}+\dim M-(r-2)d=w$.
We have $d\ge n_r+w+3\ge\dim M+w+3$.

Let us check that $g$ is a $(n_r+n_{r-1}-d-w-3)$-frimmersion.
Clearly, $g|_M$ is proper, so $g$ is proper.
Since the restrictions of $f$ to the components are framed immersions transverse to each other, we obtain that $M$ is a manifold,
and $g|_M$ is a framed immersion transverse to $f|_{N_-}$.
Thus the restrictions of $g$ to the components are transverse to each other.
The map $f|_{N_{r-1}}$ self-transverse, and $f|_{N_{r-1}}$ is transverse to $f|_{N_r}$.
This and the hypothesis on $S(f|_{N_r})$ imply that $S(g|_M)$ is the image of a self-transverse immersion of a finite disjoint union of manifolds, having dimensions at most
$$\max\{2n_{r-1}+n_r-2d,n_r-w-3+n_{r-1}-d\}=(n_r+n_{r-1}-d)-w-3,$$
where the latter equality follows from $d\ge n_r+w+3$.

The property (\ref{l:ldmin}.1) for ($(f,r)$ replaced by) $(g,r-1)$ holds because
$\partial M\subset\partial N_{r-1}\times \partial N_r$.

The property (\ref{l:ldmin}.3) for manifolds $N_1,\ldots,N_{r-2},M$ holds because $M$ is $(w+1)$-connected for $r\ge4$.


\begin{proof}[Checking the property (\ref{l:ldmin}.2) for ($(f,r)$ replaced by) $(g,r-1)$]
We need to prove that for
$$N^\times_-:=N_1\times\ldots\times N_{r-2}\times M$$
there is a map
$$\Psi:N^\times_-\to B^{d\times(r-1)}-\diag\phantom{}_{r-1}\quad\text{extending}\quad
g^{r-1}:\partial N^\times_-\to B^{d\times(r-1)}-\diag\phantom{}_{r-1}.$$
By (\ref{l:ldmin}.2) (for $(f,r)$) there is a general position homotopy
$$\Phi:N^\times\times I\to B^{d\times r}\rel\partial N^\times\quad\text{such that}\quad
\Phi_0=f^r\quad\text{and}\quad \im \Phi_1\cap \diag_r=\emptyset.$$
Identify $B^{d\times(r-1)}$ with the subset of $B^{d\times r}$ consisting of $r$-tuples of vectors
from $B^d$ for which the $(r-1)$-th and $r$-th vectors are the same.
Then as a submanifold $\delta_{r-1}$ is identified with $\delta_r$, but we use different notation for them as framed submanifolds.
The natural normal framing of $\diag_r$ in $B^{d\times r}$ is an extension of the natural normal framing of $\diag_{r-1}$ in $B^{d\times(r-1)}$
by the natural normal framing of $B^{d\times(r-1)}$ in $B^{d\times r}$.

Then $N^\times_-$ contains a framed submanifold $(f^r)^{-1}\diag_r$ of $N^\times$, which
is the framed boundary of a framed submanifold $\Phi^{-1}\diag_r$ of $N^\times$.
The framed submanifold $(g^{r-1})^{-1}\diag_{r-1}$ of $N^\times_-$ is obtained from $\partial\Phi^{-1}\diag_r=(f^r)^{-1}\diag_r$
by forgetting the part of its normal framing formed by normal framing of $N^\times_-$ in $N^\times$.
The Pontryagin-Thom construction (see e.g. \cite[\S18.5]{Pr06}) shows that for proving the existence of $\Psi$ it suffices to prove that

(*) \quad {\it $(g^{r-1})^{-1}\diag_{r-1}$ is the framed boundary of certain framed submanifold of $N^\times_-\times I$.}

\smallskip
{\it Proof of (*).}
Since $r\ge3$, by (\ref{l:ldmin}.3) (for $(f,r)$) $N^\times$ is $(w+1)$-connected.
By (\ref{l:ldmin}.3) (for $(f,r)$) and since $M$ is $w$-connected, we obtain that $N^\times_-$ is  $w$-connected.
Then by the exact sequence of pair the pair $(N^\times,N^\times_-)$ is $(w+1)$-connected.
Hence  $(N^\times\times I,N^\times_-\times I)$ is  $(w+1)$-connected.


By general position $\dim \Phi^{-1}\diag_r=w+1$.
Thus the inclusion $\Phi^{-1}\diag_r\to N^\times\times I$ is homotopic, relative to the boundary, to a map whose image is in $N^\times_-\times I$.
Since $r\ge3$, we have
$$d(2r-3)\ge dr>(r-1)d+w+1,\quad\text{so}\quad 2w+2<(r-2)d+w+1=\dim(N^\times_-\times I).$$
Thus $\Phi^{-1}\diag_r$ is isotopic, and hence framed cobordant, relative to the boundary, to a $(w+1)$-submanifold $F$ of $N^\times_-\times I$ framed in $N^\times\times I$.
In the following paragraph we check that $F$ is as required, i.e., $F$ is framed in $N^\times_-\times I$.


We have $(g^{r-1})^{-1}\diag_{r-1}=(f^r)^{-1}\diag_r=\partial F$ as framed submanifolds of $N^\times_-$.
Since $w+1<(r-2)d$,
the inclusion induces a monomorphism $\pi_j(SO_{(r-2)d})\to\pi_j(SO_{(r-1)d})$ for each $j<w+1$.
Hence the normal framing in $N^\times_-$ of $F$ decomposes by skeleta to the sum of a normal framing of $F$ in $N^\times_-\times I$
and the restriction to $F$ of the normal framing of $N^\times_-\times I$ in $N^\times\times I$.
\end{proof}



\begin{proof}[Proof of Proposition \ref{l:ldmin}: completion of the inductive step]
Thus $g$ satisfies the hypothesis of Proposition \ref{l:ldmin} for $r$ replaced by $r-1$.
Applying the inductive hypothesis to $g$ we obtain a map $\overline g:N_-\sqcup M\to B^d$ such that
$$g=\overline g\quad\text{on}\quad \partial N_-\sqcup M
\quad\text{and}\quad \overline gN_1\cap\ldots\cap \overline gN_{r-2}\cap \overline gM=\emptyset.$$
Define $\overline f:N\to B^d$ to be $\overline g$ on $N_-$ and to be $f$ on $N_{r-1}\sqcup N_r$.
Then $\overline f$ is as required.
\end{proof}

\section{The Metastable Local Realization Theorem}\label{s:metar}

Maps $\varphi,\psi:N^\times\to B^{d\times r}-\diag_r$ {\bf agree} along a homotopy
$\alpha_t:\partial N^\times\to B^{d\times r}-\diag_r$ between $\varphi|_{\partial N^\times}$ and
$\psi|_{\partial N^\times}$, if
$\alpha$ extends to a homotopy between $\varphi$ and $\psi$.

The following theorem is a (`relative link map') version of the celebrated Weber Theorem 
\cite[the Weber Theorem 8.1]{Sk06}, and is (a stronger version of) the case $r=2$ of the Metastable 
Local Realization Theorem \ref{l:ldmr} below (see also the sentence before Remark \ref{r:sphere}).

\begin{Theorem}\label{l:ldmr2} Let

$\bullet$ $f_0:N_1\sqcup N_2\to B^d$ be a proper map of a disjoint inion of $n_1$- and $n_2$-manifolds such that
$2d\ge2n_2+n_1+3$ and $f_0N_1\cap f_0N_2=\emptyset$, and

$\bullet$ $\Phi:N_1\times N_2\to B^{d\times r}-\diag_r$ a map such that
$\Phi|_{\partial(N_1\times N_2)}=f_0^2|_{\partial(N_1\times N_2)}$.

Then there is a homotopy $f_t:N\to B^d$ fixed on $\partial N_1\sqcup N_2$ such that $f_1N_1\cap f_1N_2=\emptyset$,
and the maps $f_1^2$, $\Phi$ agree along the homotopy $f_t^2$.
\end{Theorem}

{\it Comments on the proof.}
Theorem \ref{l:ldmr2} easily follows from \cite[Theorem 1.3]{Sk00} for $N=\emptyset$ and manifolds $Q=N_1$, $P=N_2$, see the details in Remark \ref{r:ifhawe}.
More precisely, we need an improvement of \cite[Theorem 1.3]{Sk00} in which `$\t{f_1}\simeq\Phi$' is replaced by
`maps $\t{f_1}$, $\Phi$ agree along the `composition'
$\Phi|_{\partial\t K}\overset{H}\simeq \t{f_0}|_{\partial\t K}\overset{\t{f_t}}\simeq \t{f_1}|_{\partial\t K}$ of given homotopy $H$ and $\t{f_t}$.'
This improvement is in fact proved in \cite[\S2]{Sk00},
cf. \cite[Realization Lemma 2.2]{Sk00} and proof of Theorem \ref{t:mmwi} in \S\ref{s:prmmw}.

\medskip
Denote by $\#$ the connected sum of maps.

\begin{Theorem}[Metastable Local Realization]\label{l:ldmr}
Let $f_0:N\to B^d$ be a $(n_1,\ldots,n_r)$-Whitney map of a stably parallelizable manifold such that $f_0N_1\cap\ldots\cap f_0N_r=\emptyset$, and $z:S^n\to B^{d\times r}-\diag_r$ a map.
Then there is a homotopy $f_t:N\to B^d$ fixed on $\partial N$ such that $f_1N_1\cap\ldots\cap f_1N_r=\emptyset$, and

(\ref{l:ldmr}.*) the maps $f_1^r$, $f_0^r\#z$ agree along the homotopy $f_t^r$.
\end{Theorem}


A {\bf pattern} is a proper framed immersion $f:N_1\sqcup\ldots\sqcup N_r\to B^d$ of a disjoint union of manifolds such that $fN_1\cap\ldots\cap fN_r=\emptyset$ and the restrictions to the components are transverse to each other.
Theorem \ref{l:ldmr} is deduced from its following version in which we assume that $f_0$ is a pattern instead of assuming that $N_i$ are stably parallelizable, and require the homotopy to be fixed outside $\Int N_1$, not only on $\partial N$.

\begin{Proposition}\label{l:ldmrin}
Let $f_0:N\to B^d$ be a $(n_1,\ldots,n_r)$-Whitney map and a pattern, and $z:S^n\to B^{d\times r}-\diag_r$ a map.
Then there is a homotopy $f_t:N\to B^d$ fixed outside $\Int N_1$ such that
$f_1N_1\cap\ldots\cap f_1N_r=\emptyset$, and (\ref{l:ldmr}.*) holds.
\end{Proposition}


\begin{proof}[Proof of Theorem \ref{l:ldmr} assuming Proposition \ref{l:ldmrin}]
Analogously to the first paragraph of the proof of the `if' part of Theorem \ref{l:ldm} (\S\ref{s:meta})
we obtain a small homotopy $F:N\times I\to B^d$ from $F_0=f_0$ to a proper framed self-transverse immersion $F_1$.
Since the homotopy $F$ is small, for every $t\in[0,1]$ we have $F_tN_1\cap\ldots\cap F_tN_r=\emptyset$, so
$F_t$ is a pattern and $F_t^r(N^\times)\cap\diag_r=\emptyset$.
Hence applying Proposition \ref{l:ldmrin} to $f_0=F_1$ we obtain a homotopy $F:N\times[1,2]\to B^d$
fixed outside $\Int N_1$ such that

(i) $F_2N_1\cap\ldots\cap F_2N_r=\emptyset$, and

(ii) the maps $F_2^r$, $F_1^r\#z$ agree along the homotopy $F_t^r$, $t\in[1,2]$.

Take a homeomorphism $h: \partial N\times I \bigcup\limits_{\partial N\times0=\partial N} N \to N$
(cf. end of the proof of the `if' part of Theorem \ref{l:ldm}).
Since the homotopy $F:N\times[1,2]\to B^d$ is fixed on $\partial N$, we have
$F_t\partial N_1\cap\ldots\cap F_t\partial N_r=F_1\partial N_1\cap\ldots\cap F_1\partial N_r=\emptyset$ for $t\in[1,2]$.
Hence we may assume that the collar $h(\partial N\times I)$ is so small that

(iii) $F_th(\partial N_1\times I)\cap\ldots\cap F_th(\partial N_r\times I)=\emptyset$ for every $t\in[1,2]$.

Define the required homotopy
$$f:N\times[0,1]\to B^d\quad\text{by}\quad
f(x,t)=\begin{cases} F_{2t}(x) & x\in h(N) \\
F_{4(1-s)}h(y,1) & x=h(y,s)\in h(\partial N\times I),\ s\ge1-\frac t2 \\
F_{2t}h\left(y,\dfrac s{1-\frac t2}\right)& x=h(y,s)\in h(\partial N\times I),\ s\le1-\frac t2
\end{cases}.$$
The meaning of this formula on $h(\partial N\times I)$ is that $f(h(y,s),t)$ for $s\le1-\frac t2$
changes from $F_{2t}h(y,0)$ to $F_{2t}h(y,1)$, and then for $s\le1-\frac t2$ changes to $F_0h(y,1)$.

Clearly, $f_0=F_0$ indeed coincides with the given map $f_0$.

Since $f(h(y,1),t)=F_0(h(y,1))$, the homotopy $f_t$ is fixed on $\partial N$.

The map $f_1$ is the `union' of $F_2$, $F_2|_{h(\partial N\times I)}$ and
$F|_{h(\partial N\times I)\times[0,2]}$.
This, (i), and (iii) imply that $f_1N_1\cap\ldots\cap f_1N_r=\emptyset$.

By (ii), the maps $(f_1h^{-1})^r$, $(f_{1/2}h^{-1})^r\#z$ agree along the homotopy
$(f_th^{-1})^r$, $t\in[1/2,1]$.
Also $f_t^r(x)\not\in\delta_r$ for every $(x,t)\in N^\times\times[0,1]-h^r(\Int N^\times)\times[1/2,1]$.
Hence (\ref{l:ldmr}.*) holds.
\end{proof}

\begin{proof}[Proof of Proposition \ref{l:ldmrin}]
By Remark \ref{r:sphere}.b there is a homotopy equivalence $B^{d\times r}-\diag_r\to S^{d(r-1)-1}_{\Sigma_r}$.
Proposition \ref{l:ldmrin} does not involve the action of $\Sigma_r$ on $S^{d(r-1)-1}_{\Sigma_r}$.
So it suffices to prove Proposition \ref{l:ldmrin} for $B^{d\times r}-\diag_r$ replaced by $S^{d(r-1)-1}$.
Below we refer to Proposition \ref{l:ldmrin} with such a replacement.

We have $\dim(\partial N^\times\times I)=n$.
Thus for $n<d(r-1)-1$ Proposition \ref{l:ldmrin} holds for the identical homotopy $f_t$ by general position.
So assume that $n\ge(r-1)d-1$.
Then $d=rd-(r-1)d\ge n+n_i+3-n-1=n_i+2$.

The proof is by induction on $r$.\aronly{\footnote{Instead of the induction, we can reduce the case of arbitrary $r$ to the case $r=2$ by setting $M:=\{(x_2,\ldots,x_r)\in N_2\times\ldots\times N_r\ :\ fx_2=\ldots=fx_r\}$, and
defining $g:N_1\sqcup M\to B^d$ to be $f$ on $N_1$ and $f\circ\pr_{N_2}$ on $M$.
Then we apply $r-2$ times the analogues of the argument from `Proof of (\ref{l:ldmr}.*)'
below.}}
The base $r=2$ follows from Theorem \ref{l:ldmr2}; we do not need $f$ to be a pattern.
Let us prove the inductive step $r-1\to r\ge3$.

As in the proof of Proposition \ref{l:ldmin}, define
$$N_-:=N_1\sqcup\ldots\sqcup N_{r-2},\quad
M=M(f_0|_{N_{r-1}\sqcup N_r})=\{(x,y)\in N_{r-1}\times N_r\ :\ f_0x=f_0y\},$$
$$N^\times_-:=N_1\times\ldots\times N_{r-2}\times M,$$
define $g:N_-\sqcup M\to B^d$ to be $f$ on $N_-$, and $f\pr_{N_r}$ on $M$,
and prove that $g$ is a $(n_1,\ldots,n_{r-2},n_{r-1}+n_r-d)$-Whitney map and a is pattern.


We have
$$d(2r-3)\ge dr>n+3 \quad\Rightarrow\quad 2d(r-2)-3\ge n-d.$$
Then the suspension $\Sigma^d:\pi_{n-d}(S^{d(r-2)-1})\to \pi_n(S^{d(r-1)-1})$ is an epimorphism.
So by the Freudenthal Suspension Theorem there is a map $u:S^{n-d}\to S^{d(r-2)-1}$ such that $\Sigma^d[u]=[z]$.

Apply the inductive hypothesis to the pattern $g_0=g$ and the map $u$.
We obtain a homotopy $g_t:N_-\sqcup M\to B^d$ fixed outside $\Int N_1$ such that
$g_1N_1\cap\ldots\cap g_1N_r=\emptyset$ and

(**) the maps $g_1^{r-1}$, $g_0^{r-1}\#u$ of $N^\times_-$ agree along the homotopy $g_t^{r-1}$ on $\partial N^\times_-$.

Let $f_t:N\to B^d$ be $g_t$ on $N_-$ and $f_0$ on $N_{r-1}\sqcup N_r$.
Clearly, $f_1N_1\cap\ldots\cap f_1N_r=\emptyset$.

\begin{figure}[h]
\centerline{\includegraphics
{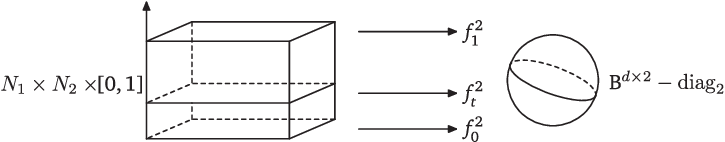}}
\caption{The map $F$ for $r=2$; $\mbox{diag}_2=\diag_2$.}
\label{f-hff}
\end{figure}

\smallskip
{\it Proof of (\ref{l:ldmr}.*).}
Define the map (see Figure \ref{f-hff})\aronly{\footnote{Since $f_0N_1\cap\ldots\cap f_0N_r=\emptyset$, we have $f_0^rN^\times\cap\diag_r=\emptyset$.
Analogously $f_1^rN^\times\cap\diag_r=\emptyset$.
The map $F$ is `the difference between $f_1^r$ and $f_0^r$ on $N^\times$ via homotopy $f_t^r$ on
$\partial N^\times$'.}}
$$F:\partial(N^\times\times[0,1])\to S^{d(r-1)-1}\quad\text{by}\quad F(x,t)=f_t^r(x).$$
Denote by $c$ the constant map (between any spaces).
Then (\ref{l:ldmr}.*) is equivalent to $F\#z$ being homotopic to $c$, and hence to $F$ being homotopic to $c\#(-z)$.
Here by $-z$ we denote any representative of the inverse of the homotopy class of $z$;
analogous meaning has $-u$ below.
Analogously define the map
$$G:\partial(N^\times_-\times[0,1])\to S^{d(r-2)-1}\quad\text{by}\quad G(x,t)=g_t^r(x).$$
Then (**) is equivalent to $G$ being homotopic to $c\#(-u)$.

For each manifold $X$ and integer $m$ identify by the Pontryagin-Thom 1--1 correspondence \cite[\S18.5]{Pr06}
the set $[X;S^m]$ with the set of codimension $m$ framed submanifold of $X$ up to framed cobordism.
Identify $B^{d\times(r-1)}$ with the subset of $B^{d\times r}$ consisting of $r$-tuples of vectors
from $B^d$ for which the $(r-1)$-th and $r$-th vectors are the same.
Then $S^{d(r-2)-1}$ is identified with certain equatorial sphere in $S^{d(r-1)-1}$.
Shift the homotopy $g_t$ to general position modulo $g_0$ and $g_1$.

Take a regular value $a\in S^{d(r-2)-1}$ of $G$.
Since the intersection $fN_{r-1}\cap fN_r$ is transversal, $a$ is a regular value of $F$.
The natural normal framing of $a$ in $S^{d(r-1)-1}$ is the extension of the natural normal framing of $a$
in $S^{d(r-2)-1}$ by the natural normal framing of $S^{d(r-2)-1}$ in $S^{d(r-1)-1}$.
The framed preimages $F^{-1}a$ and $G^{-1}a$ are framed submanifolds of $\partial(N^\times\times[0,1])$ and of $\partial(N^\times_-\times[0,1])$, respectively.
We have $f_t^r|_{N_-}=g_t^r$.
Hence the supporting manifolds of $F^{-1}a$ and of $G^{-1}a$ are the same, and the normal framing
of $F^{-1}a$ is the extension of the normal framing of $G^{-1}a$ by the natural normal framing
of $S^{d(r-2)-1}$ in $S^{d(r-1)-1}$.
Since $G$ is homotopic to $c\#(-u)$, we may assume that $G^{-1}a$ lies in an $(n-d)$-ball in $\partial(N^\times_-\times[0,1])$.
Thus the map $F$ is homotopic to $c\#\Sigma^d(-u)=c\#(-z)$.
Hence (\ref{l:ldmr}.*) holds.
\end{proof}

\section{Disjunction and realization for a complex}\label{s:prmmw}

The definition of agreement is given analogously to the beginning of \S\ref{s:metar}
for $\Sigma_r$-equivariant maps and homotopies from $\Sigma_r$-equivariant cell complexes $E_0\subset E_1$,
and for $B^{d\times r}$ replaced by $\R^{d\times r}$.


\begin{Theorem}[cf. {\cite[Disjunction Theorem 3.1]{Sk02}}]\label{t:mmwi}
Assume that $rd\ge(r+1)k+3$,

$\bullet$ $K$ is a $k$-complex,


$\bullet$
$E_0\subset E_1 \subset K^{\times r}_{\Delta}$ are $\Sigma_r$-equivariant subsets of $K^r$ (see Figure \ref{f-e0e1}) such that if $j\in\{0,1\}$ and
\linebreak
$\Int(\sigma_1\times\cdots\times\sigma_r)\cap E_j\ne\emptyset$, then $\sigma_1\times\cdots\times\sigma_r\subset E_j$,

$\bullet$ $f_0:K\to \R^d$ is a general position map such that
$f_0^rE_0\cap\diag_r=\emptyset$,

$\bullet$ $\Phi:E_1\to \R^{d\times r}-\diag_r$ is a $\Sigma_r$-equivariant map,

$\bullet$ $H$ is a $\Sigma_r$-equivariant homotopy from $\Phi|_{E_0}$ to $f_0^r|_{E_0}$.


Then there is a general position homotopy $f_t:K\to\R^d$
such that


(\ref{t:mmwi}.1) $f_1^rE_1\cap\diag_r=\emptyset$,

(\ref{t:mmwi}.2) $f_t^rE_0\cap\diag_r=\emptyset$ for each $t$, and

(\ref{t:mmwi}.3) the maps $\Phi$, $f_1^r|_{E_1}$ agree along
the `composition' $\Phi|_{E_0}\overset{H}\simeq f_0^r|_{E_0}\overset{f_t^r}\simeq f_1^r|_{E_0}$.
 \end{Theorem}


\begin{figure}[h]
\centerline{\includegraphics
{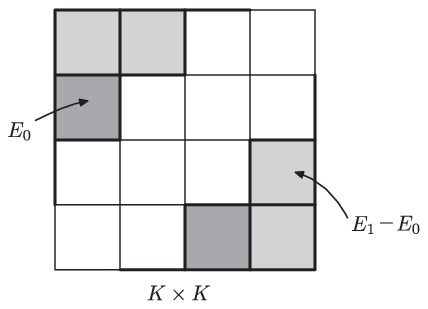}}
\caption{The sets $E_0$ and $E_1$ for $r=2$.}
\label{f-e0e1}
\end{figure}

In this section we prove Theorem \ref{t:mmwi}.
For simplices $\sigma_1,\ldots,\sigma_r$ of a complex denote
$$\sigma^\times:=\sigma_1\times\ldots\times\sigma_r,\quad n_i:=\dim\sigma_i,\quad\text{and}\quad n=n_1+\ldots+n_r.$$


\begin{Lemma}\label{l:ball} Assume that $K$ is a $k$-complex, $\sigma_1,\ldots,\sigma_r$ are
its pairwise disjoint simplices, $f:K\to \R^d$ is a general position PL map such that
$f^r\partial\sigma^\times\cap\diag_r=\emptyset$, and $rd\ge(r+1)k+3$.
Then there is a PL $d$-ball $\beta\subset\R^d$ such that

(\ref{l:ball}.1) $D^{n_i}\cong N_i:=\sigma_j\cap f^{-1}\beta\subset\Int\sigma_j$
and $f|_{N_i}$ is proper for each $i=1,\ldots,r$;

(\ref{l:ball}.2) $f\sigma_1\cap\ldots\cap f\sigma_r\subset\Int\beta$;

(\ref{l:ball}.3) $f^{-1}\beta$ is a regular neighborhood in $K$ of $N:=N_1\sqcup\ldots\sqcup N_r$
relative to $\partial N$.\aronly{\footnote{This implies that $f^{-1}\beta\subset\st(\Int\sigma_1\sqcup\ldots\sqcup\Int\sigma_r)$, i.e., $\tau\cap f^{-1}\beta=\emptyset$ when $\tau$ is a simplex of $K$ such that $\tau\not\supset\sigma_i$ for each $i=1,\ldots,r$.}}
\end{Lemma}

\begin{proof}
The proof is analogous to \cite[\S5]{We67}, see also \cite[Lemma 8]{MW16},
\jonly{\cite[Remark 6.1.b9]{Sk17} and references of \cite[footnote 6]{Sk17}.}
\aronly{Remark \ref{r:sp}.b9 and references of footnote \ref{f:disj}.}
The definitions of a collapsible polyhedron, of a regular neighborhood and of the trace of a collapse can be found e.g. in \cite[\S8.3]{Sk06}.

Denote $Z:=f\sigma_1\cap\ldots\cap f\sigma_r$.
By general position $\dim Z\le n-(r-1)d$.
Take any $i=1,\ldots,r$.
Since $f^r\partial\sigma^\times\cap\diag_r=\emptyset$, we have $Z\cap f\partial\sigma_j=\emptyset$.
Hence there is a collapsible polyhedron $C_i\subset\Int\sigma_i$ of dimension at most $n-(r-1)d+1$ containing $\sigma_i\cap f^{-1}Z$.\aronly{\footnote{We can take $C_i$ to be the trace of $\sigma_i\cap f^{-1}Z$ under a collapse of $\sigma_i$ to a point \cite[\S8.4]{Sk06}, thus avoiding use of a more complicated Zeeman's engulfing lemma \cite[Lemma 7]{MW16}. Analogous remark could be made for the construction of $C$ below.}}
Since
$$n\le kr\quad\text{and}\quad rd\ge(r+1)k+3,\quad\text{we have}\quad 2k+n-(r-1)d+2<2d.$$
Hence by general position we may assume that $C_i\cap S(f)=\emptyset$.

There is a collapsible polyhedron $C\subset\R^d$ of dimension at most $n-(r-1)d+2$ containing
$fC_1\cup \ldots\cup fC_r$.
Since
$$n\le kr\quad\text{and}\quad rd\ge(r+1)k+3,\quad\text{we have}\quad k+n-(r-1)d+2<d.$$
Hence by general position we may assume that $C\cap fK=fC_1\cup\ldots\cup fC_r$ and $C\cap f\sigma_i=fC_i$ for each $i=1,\ldots,r$.
Let us prove that the regular neighborhood $\beta$ of $C$ in $\R^d$ is the required PL $d$-ball.


Since $C\supset Z$, the property (\ref{l:ball}.2) holds.

Since $C\cap f\sigma_i=fC_i$ and $C_i\cap S(f)=\emptyset$, we have $\sigma_i\cap f^{-1}C=C_i$.
Hence $N_i$ is a regular neighborhood of $C_i$ in $\sigma_i$.
This and the collapsibility of $C_i\subset\Int\sigma_i$ imply the property (\ref{l:ball}.1).


Since $C\cap fK=fC_1\cup\ldots\cup fC_r$ and $C_i\cap S(f)=\emptyset$, we have
$f^{-1}C=C_1\sqcup\ldots\sqcup C_r$.
Hence $f^{-1}\beta$ is a regular neighborhood in $K$ of $C_1\sqcup\ldots\sqcup C_r$.
Thus the property (\ref{l:ball}.3) holds.
\end{proof}

\begin{proof}[Proof of Theorem \ref{t:mmwi}]
{\it Preliminaries.}
By induction it suffices to prove Theorem \ref{t:mmwi} for $E_1-E_0$ being the union of the interiors of products obtained by all permutations from the product $\sigma^\times$.
We assume this in the proof.
Abbreviate $f_0$ to $f$.

Since $\partial\sigma^\times\subset E_0$ and $f_0^rE_0\cap\delta_r=\emptyset$,
we can apply Lemma \ref{l:ball}.
We obtain $\beta$ and $N_1,\ldots,N_r$.
Denote
$$N^\times:=N_1\times\ldots\times N_r \quad\text{and}\quad A:=\sigma^\times-\Int N^\times.$$
Then $f^rA\cap\delta_r=\emptyset$.
By the PL Annulus Theorem $\partial\sigma_i$ is a strong deformation retract of $\sigma_i-\Int N_i$.
Hence $\partial\sigma^\times$ is a strong deformation retract of $A$.
Then the restriction $f^r:E_0\cup A\to \R^{d\times r}-\diag_r$ is homotopic to $\Phi|_{E_0\cup A}$.
So by the Equivariant Borsuk Homotopy Extension Theorem (cf. footnote \ref{f:bhet}) there is a
$\Sigma_r$-equivariant homotopy
$E_1\times I\to\R^{d\times r}-\diag_r$ from $\Phi$ to a map whose restriction to $E_0\cup A$ coincides with $f^r$.
So it suffices to prove Theorem \ref{t:mmwi} under the additional assumptions that $\Phi=f^r$ on $E_0\cup A$, and $H$ is the identical homotopy.

\smallskip
{\it Local disjunction.}
Then the map $f^r:\partial N^\times\to\R^{d\times r}-\diag_r$ extends over $N^\times$
(to the map $\Phi|_{N^\times}$).
This and $rd\ge(r+1)k+3$ imply that we can apply the Metastable Local Disjunction Theorem \ref{l:ldm}
for disks $N_i$ and $B^d=\beta$.
We obtain an extension $\overline f:N\to\beta$ of $f|_{\partial N}$ such that
$\overline fN_1\cap\ldots\cap\overline fN_r=\emptyset$.


\smallskip
{\it Local realization.}
Join $f|_N$ and $\overline f$ by a homotopy $f_t$, $t\in[0,1/2]$, fixed on $\partial N$.
Analogously to the above we may assume that $\Phi=f_{1/2}^r$ on $E_0\cup A$.
Define a map
$$z:\partial(N^\times\times[0,1/2])\to\R^{d\times r}-\diag_r\quad\text{by}\quad
z|_{N^\times\times0}=\Phi,\ \ z|_{N^\times\times1/2}=f_{1/2}^r
\ \text{ and }\ z|_{\partial N^\times\times t}=f^r_t.$$
Since $N_i\cong D^{n_i}$ for every $i=1,\ldots,r$, we have $\partial(N^\times\times[0,1/2])\cong S^n$.
So $z$ can be considered as a map $S^n\to \R^{d\times r}-\diag_r$.
Since $rd\ge(r+1)k+3$, we can apply the Metastable Local Realization Theorem \ref{l:ldmr} for disks $N_i$, $B^d=\beta$, and taking $f_0$ to be $f_{1/2}$.
We obtain a homotopy $f_t$, $t\in[1/2,1]$, for which the maps $f_{1/2}^r\#z$, $f_1^r$ of $N^\times$ agree along the homotopy $f^r_t$ on $\partial N^\times$, $t\in[1/2,1]$.
Hence the maps $\Phi|_{N^\times}$, $f_1^r$ of $N^\times$ agree along the homotopy $f^r_t$ on $\partial N^\times$, $t\in[0,1]$.

\smallskip
{\it Completion.}
Extend the homotopy $f_t:N\to\beta$, $t\in[0,1]$, arbitrarily to a homotopy $f_t:f^{-1}\beta\to\beta$ fixed on $f^{-1}\partial\beta$.
Extend the obtained homotopy identically to $K-f^{-1}\beta$.
Denote the obtained extension by the same letter $f_t$.
Below we check the property (\ref{t:mmwi}.2).
Then the property (\ref{t:mmwi}.1) is equivalent to $f_1N_1\cap\ldots\cap f_1N_r=\emptyset$, which is `achieved' (for $f_1$ replaced by $\overline f=f_{1/2}$) at the `local disjunction' step and `kept' at the `local realization' step.
The property (\ref{t:mmwi}.3) (is possibly destroyed at the `local disjunction' step but) is `achieved' at the `local realization' step.
So $f_t$ is as required.


\smallskip
{\it Checking the property (\ref{t:mmwi}.2).}
Take any cell $\tau^\times:=\tau_1\times\ldots\times\tau_r\subset E_0$.
The cell $\tau^\times$ does not contain any cell obtained from $\sigma^\times$ by permutation of factors.
We have $\tau_j\cap\tau_k=\emptyset$ for each $j\ne k$. 
Then there is $j$ such that $\tau_j\not\supset\sigma_i$ for each $i$.
Hence by (\ref{l:ball}.3) $\tau_j\subset K-f^{-1}\beta$.
So $f_t=f$ on $\tau_j$.
Thus $f_t\tau_1\cap\dots\cap f_t\tau_r\subset\Int\beta\cap(\R^d-\beta)=\emptyset$.
\end{proof}

\aronly{

\begin{Remark}\label{r:conj}
(a) In Theorem \ref{l:ldm} if $f$ is piecewise linear (PL) / smooth on $\partial N$ (in the PL case assume that $B^d:=[0,1]^d$), then we may assume that $\overline f$ is PL / smooth (by approximation).

(b) Relative version of Theorem \ref{t:mmwi} perhaps allows
classification of ornaments, cf. \cite[\S1.3]{AMSW}.

(c) Assume that $rd\ge(r+1)k+4$ and $K$ is a $k$-complex.
We conjecture that the map $f\mapsto f^r$ defines a 1--1 correspondence between almost $r$-embeddings $K\to\R^d$ up to homotopy through almost $r$-embeddings, and $\Sigma_r$-equivariant maps
$K^{\times r}_{\Delta}\to \R^{d\times r}-\diag\phantom{}_r$ up to $\Sigma_r$-equivariant homotopy.

In this conjecture the surjectivity holds by Theorem \ref{t:mmwi} below (for $E_0=\emptyset$
and $E_1=K^{\times r}_{\Delta}$).
The injectivity can perhaps be proved using a relative version of Theorem \ref{t:mmwi} and the idea of \cite[\S5]{Sk02}.

(d) The connectedness and the stable parallelizability assumptions in Theorems \ref{l:ldm} and \ref{l:ldmr} can perhaps be omitted; these results presumably hold even for polyhedra not manifolds $N_i$).
Such more general versions could perhaps be proved using Theorem \ref{t:mmwi} (whose proof would still use current versions Theorems \ref{l:ldm} and \ref{l:ldmr} proved independently of Theorem \ref{t:mmwi}).

(e) Our proof can perhaps be modified to show that in Theorem \ref{l:ldm} (and in Propositions \ref{l:ldmin}) when and $f|_{N_i}$ is an embedding for each $i$, we may obtain additionally that $\overline f|_{N_i}$ is an embedding isotopic to $f|_{N_i}$, relative to the boundary, for each $i$, cf. \cite[Lemma~10]{MW15}.
Analogous modifications can perhaps be done in Theorem \ref{l:ldmr} and in Proposition \ref{l:ldmrin} when
$d-n_i\ge3$ (this codimension 3 assumption is automatic for Theorem \ref{l:ldm} but not for Theorem \ref{l:ldmr}: take $d=2r-1=n_i+2$).
For Propositions \ref{l:ldmin} and \ref{l:ldmrin} this is presumably a standard application of the corresponding `injective' version of the Surgery of Intersection Lemma \ref{l:surg}, together with the Haefliger-Zeeman Unknotting Theorems, cf. \cite[\S4, proof of Proposition 2 presented before Proposition 2]{We67},
\cite[proofs of Propositions 8.4 and 8.5, a generalization of the Whitney trick]{Sk06}.
\end{Remark}

}

\section{Appendix: proof of the Surgery of Intersection Lemma}\label{s:apprsuin}

By the transversality of $f|_P$ and $f|_Q$,
the set $M(f)$ is a $(p+q-d)$-manifold.
Take any $j\in\{0,1,2,\ldots,k\}$.
Assume additionally that $M(f)$ is $(j-1)$-connected and $u:S^j\to M(f)$ is a map.
It suffices to construct a
$(q-k-2)$-frimmersion $f':P\sqcup Q\to B^d$ such that $f'=f$ on $\partial P\sqcup Q$, and $M(f')$ is $(j-1)$-connected and
$$(*)\qquad\pi_i(M(f'))\cong\pi_i(M(f))\quad\text{for each}\quad i<j\quad\text{and}\quad \pi_j(M(f'))\cong\pi_j(M(f))/\left<u\right>.$$
Here $\left<u\right>$ is certain (for $j=1$ normal) subgroup containing the homotopy class of $u$.
For $j=0$ the latter condition is replaced by `$M(f')$ has a smaller number of connected components than $M(f)$'.

By the self-transversality of $f|_P$, the set $S(f|_P)$ is the image of an immersion of a $(2p-d)$-manifold.
This and the hypothesis on $S(f|_Q)$ imply that $S(f\pr_Q:M(f)\to B^d)$ is the image of a self-transverse immersion
of a finite disjoint union of manifolds of dimensions at most
$$\max\{2p+q-2d,q-k-2+p-d\}\le p+q-d-k-1,$$
where the inequality follows from $k\le d-p-1$.
This and $2j\le2k<p+q-d$ by general position\aronly{\footnote{I.e., by the following result: {\it if $M$ is an    $n$-manifold, $u:S^j\to M$ a map, $2j<n$ and $X\subset M$ is the image of a self-transverse immersion of a finite disjoint union of manifolds of dimensions at most $n-j-1$, then $u$ homotopic to an embedding $S^j\to M$ whose image misses $X$}.
Other general position results of this kind are used below without explicit statement.}}
imply that $u$ is homotopic to an embedding whose composition with $f\pr_Q$ is an embedding into $B^d$.
For brevity, denote the embedding by the same letter $u$.
Then the compositions $\pr_Pu$ and $\pr_Qu$ are also embeddings $S^j\to P$ and $S^j\to Q$.

By the transversality of $f|_P$ and $f|_Q$ the sphere $\pr_PuS^j$ has a normal $(d-q)$-framing in $P$.
Push the sphere along the first vector field of this framing.
We obtain an embedding into $P$ of the product of the sphere with the interval.
We have
$$j\le k,\quad k+1+(2p-d)<p,\quad k+1+(q-k-2)<q\quad\text{and}\quad 2k+2\le p+q-d\le p,q.$$
Thus by general position the embedding extends to an embedding  $u_P:D^{j+1}\to P$  whose composition with $f$ is  an embedding into $B^d$.
Since $d-k-2\ge p$, we have $2d\ge 2p+2k+4>2p+k+1$.
Hence by general position $u_PD^{j+1}\cap S(f|_P)=\emptyset$.

Analogously we construct an embedding $u_Q:D^{j+1}\to Q$  whose composition with $f$ is  an embedding into $B^d$.
Since $d-k-2\ge q$, we have $d\ge q+k+2>q-k-2+k+1$.
Hence by general position $u_QD^{j+1}\cap S(f|_Q)=\emptyset$.

We have $fu_PD^{j+1}\cap fu_QD^{j+1}=fu_PS^j=fu_QS^j\cong S^j$.
Push $fu_PD^{j+1}$ and $fu_QD^{j+1}$ along the first vector field of the normal framing of $f$.
We obtain an embedding into $B^d$ of the product of the $(j+1)$-sphere $fu_PD^{j+1}\cup fu_QD^{j+1}$ with the interval (with corners along $fu_PS^j=fu_QS^j$).
We have
$$j\le k,\ d-k-2>p,q\text{ and }2(k+2)-d\le p+q-2d+2\le2(d-k-2)-2d+2=-2(k+1)<0,$$
Thus by general position the embedding extends to an embedding $\overline u:D^{j+2}\to B^d$ (with corners along $fu_PS^j=fu_QS^j$) such that $\overline uD^{j+2}\cap f(P\sqcup Q)=fu_PD^{j+1}\cup fu_QD^{j+1}$.


Denote by $V_{m,n}$ the Stiefel manifold of orthonormal $n$-frames in $\R^m$.
Since $2j\le 2k<p+q-d-1$, we have $j+d-q-1<p-j-1$, so $\pi_j(V_{p-j-1,d-q-1})=0$.
Hence the normal $(d-q-1)$-framing of $u_PS^j$ in $P$ extends to a normal $(d-q-1)$-framing of $u_PD^{j+1}$ in $P$.
Take normal $(d-p-1)$-framing of $fu_PD^{j+1}$ in $B^d$ obtained by omitting the first vector field of the normal framing of $f$.
Take analogous normal $(d-p-1)$-framing of $u_QD^{j+1}$ in $Q$ extending the normal $(d-p-1)$-framing of $u_QS^j$ in $Q$, and normal $(d-q-1)$-framing of $fu_QD^{j+1}$ in $B^d$.
Since $2j\le 2k<p-2$, we have $j+1+d-p-1<d-j-2$, so $\pi_j(V_{p-j-1,d-q-1})=0$.
So the normal $(d-q-1)$-framing of $\overline uS^{j+1}$ in $B^d$ extends to a normal $(d-q-1)$-framing of $\overline uD^{j+2}$ in $B^d$.
Analogously the normal $(d-p-1)$-framing of $\overline uS^{j+1}$ in $B^d$ extends to a normal $(d-p-1)$-framing of $\overline uD^{j+2}$ in $B^d$.

Thus the relevant framing information along $\overline uD^{j+2}$ agrees with that of the following standard model.
Represent
$$\R^d=\R\times\R^{j+1}\times\R^{p+q+d-j}\times\R^{d-p-1}\times\R^{d-q-1},$$
$$\R^p=\R^{j+1}\times\R^{p+q+d-j}\times\R^{d-q-1}\quad\text{and}\quad \R^q=\R^{j+1}\times\R^{p+q+d-j}\times\R^{d-p-1}.$$
Take two embeddings $g_P:\R^p\to\R^d$ and $g_Q:\R^q\to\R^d$ intersecting transversally along
$0\times S^j\times\R^{p+q-d-j}\times0\times0$.
For example, one may take $g_Q(x, y,b): =(|x|^2 - 1, x, y, 0, b)$ and $g_P(x, y, a) = (1 - |x|^2, x, y, a, 0)$. The sphere $S^j$ bounds the ball $D^{j+1}\subset\R^{j+1}\subset\R^p,\R^q$.
Furthermore, $g_PD^{j+1}\cup g_QD^{j+1}\cong S^{j+1}$ with corners along $0\times S^j\times 0\times0\times0$.
This $(j+1)$-sphere bounds a $(j+2)$-ball in $\R\times\R^{j+1}\subset\R^d$.
So there is a diffeomorphism between a neighborhood of $\overline D^{j+2}$ and $\R^m$ taking a restriction of $f$ to $g_P\sqcup g_Q$.

Change $g_PD^{j+1}$
by pushing across the $(j+2)$-ball.
Let $f'$ be the map obtained by the inverse diffeomorphism.
Then $M(f')$ is obtained from $M(f)$ by surgery killing the spheroid $u$.
(Indeed, the double points of the `pushing' regular homotopy form the trace of this surgery.)
Since $2k+2\le p+q-d$, this is surgery below the middle dimension in the sense of \cite{Mi61}, so (*) holds.\aronly{\footnote{For this reason in \cite[Appendix A]{HK98} one needs to replace `$2s<\dim D$' by `$2s+1<\dim D$' as in  \cite[Appendix, (ii)]{CRS}.}}
Since $u_PD^{j+1}\cap S(f|_P)=u_QD^{j+1}\cap S(f|_Q)=\emptyset$, the map $f'$ a $(q-k-2)$-frimmersion.
Thus $f'$ is as required.




\aronly{

\section{Appendix: on the Mabillard-Wagner paper}\label{s:apmawa}

Metastable Local Disjunction Theorem \ref{l:ldm} and Metastable Mabillard-Wagner Theorem \ref{t:mmw}
are claimed in \cite{MW16, MW16'}
(Theorem \ref{l:ldm} was claimed only for disks $N_i$ and in a slightly different equivalent formulation).
However, because of Remark \ref{r:sp} below I do not find \cite{MW16} a reliable reference for this result
(see  \cite[p. 2]{Sk21d} for the explanation of what is meant by a `reliable reference' here;
version 1 of \cite{MW16} and the paper \cite{MW16'} contains even less details.)
See also Remark \ref{r:imm}.b.

Uli Wagner and I find open publication of this criticism of \cite{MW16, MW16'} important to stimulate appearance of a reliable reference.
See also \cite[Examples 3.2 and 3.3]{Sk21d}.
This section might also be interesting as an example of an open discussion of a controversial question,
carried in full mutual respect of participants of the discussion.
This should be compared to Remark \ref{r:uwn} and \cite[\S6]{Sk21d} exhibiting an example of misuse of the anonymous peer review system to promote an opinion which does not stand an open discussion (and thus violating principles of scientific discussions recalled \cite[Remark 4.2.b]{Sk21d}).


In 2017 we discussed with I. Mabillard and U. Wagner the criticism of Remark \ref{r:sp} below (excluding the `added later' parts) \cite{Sk17o}.
I asked them if they agree that {\it the proof of the main results in \cite{MW16} is} incomplete {\it in the following sense.}\footnote{This notion of `incomplete' roughly corresponds to more common `not a reliable reference', see \cite[p. 2]{Sk21d}. Unfortunately, shorter formulations of the conclusion of Remark \ref{r:sp} were found to be potentially misleading.
No other meaning of `{\it incomplete}' is meant here.
In particular, I have nothing against publication of {\it conjectures} with {\it incomplete} proofs.
And I do not mean that the proof of \cite{MW16} cannot be corrected.}
I call a proof {\it incomplete} if one mathematician should be able to expect from another

(1) to wait for another (`complete') proof {\it before} using results having such a proof;

(2) to recommend, as a referee, a revision (based on specific comments) {\it before} recommending publication of  results having such a proof;

(3) to work more on such a proof (in particular, send the text privately to a small number of mathematicians working on related problems), {\it before} submitting the text to a refereed journal or to arxiv.

I received no answer to the above question.
However, I received in 2017 Remark \ref{r:uw} below \cite{Sk17o}.
It first very politely agrees with (b1)-(b7) of Remark \ref{r:sp}.
The rest of Remark \ref{r:uw} is not relevant to discussing whether the proof in \cite{MW16} is {\it incomplete} in the above sense, because it discusses a non-existent update of \cite{MW16} (by making some suggestions on that update).
Cf. \cite[Remark 2.3.abcd]{Sk21d}.
However, since U. Wagner looked at most of my comments in footnotes and suggested to publish Remark \ref{r:uw},
I am glad to include it here.

\begin{Remark}\label{r:sp}
(a) Besides the main result of \cite{MW16}, I do not know any result\footnote{Except possibly a reference kindly sent to me by I. Mabillard.}
whose formulation does not use block bundles theory \cite{RS68} but which could be proved using this theory, and whose proof not using this theory is unknown or hard.
(Although Rourke and Sanderson might have been aware of such results while they wrote \cite{RS68}.)
For me it was always hard to apply block bundles theory for such a result, and it was easier to use other means,
even if the result was invented by guessing that in a given situation PL manifolds behave analogously to smooth
ones, see \cite[\S4]{Sk02}, \cite[Remark 21]{MW16},
\cite{CS} and this paper.
So application of block bundles theory is a delicate part of the proof.
The corresponding Lemmas 16, 19 and 24 \cite{MW16} are not rigorously stated and proved, see specific remarks
in (b1)-(b9) below.
These lemmas are important for  the proof.
Accurate statements of the lemmas would be more technical.

Checking proofs and use of the lemmas in their accurate statements would become an important task
for the authors.
As always, in performing this task new problems might or might not be discovered
(an example of such a discovery is (b8)).
There is no way to see how important or negligible these problems are, except authors doing this work and looking critically on the new text.

(b1) P. 25, proof of Lemma 10, First Part.  Theorem 7 is not applicable because Proposition 13 does not
assert that $\sigma_i\cap\sigma_r$ is a PL manifold.
This assumption is tacitly used in the proof of Proposition 13 (i.e. of Lemma 14, \S4.1) but never checked.
This assumption is non-trivial to check, cf. footnote 11 in version 2 of \cite{AMSW}.

(b2) In Lemma 16 {\it unknottedness}, i.e. the existence of a homeomorphism (or an isotopy), is tacitly replaced by {\it equality}; we know it is dangerous to identify isomorphic objects; cf. use of `$\cong$' not `$=$' in the analogous situations in Lemma 19(2) and in Lemma 20;

(b3) In Lemma 16 `we can assume' is way too informal for this delicate part of the proof;

(b4) In Lemma 16 it is not clear where the restatement of the transversality (the second `i.e.') ends: at `in $B^d$' or at the end of the display formula;

(b5) In the proof of Lemma 16 `follows by Theorem 66' is not clear because no block bundle (required for application of Theorem 66) is given in the statement of Lemma 16;

(b6) In Lemmas 19, 20 and 24 `In the situation given by Lemma NN' is unclear because it is not indicated whether the assumption or the conclusion of Lemma NN is meant;

(b7) In the proof of Lemma 19 `the first property follows from Theorem 70 $<$presumably meaning Proposition 70$>$ from Appendix A' is not clear because Proposition 70 does not assert any triviality (or the existence of any homeomorphism, which the triviality means).

(b8) Observation 18 should be turned into a formal statement having formal proof.
Otherwise in Observation 18 `Since $k\le s-(r-1)d$' is not clear because
it is not written what are the hypothesis of Observation 18, and in the first line of the proof of Lemma 19
`follows by ... the above observation' is not clear because it is not written which exactly statement is meant by
`the above observation'.
This done, one would see that one needs to prove that `if $2k+1\le s_i+s_r-d$, then the unstated conclusion of Observation 18 holds' (currently this part of the proof is ignored).
Then one would see that one needs the restriction $2k+2\le s_i+s_r-d$, but  not $2k+1\le s_i+s_r-d$.
(Added later: unless one wants to show that methods of \cite{KM63}, not only of the cited paper \cite{Mi61},
are applicable.)

(b9) The statement 8.3 meaning $B^d\cap fK=B^d\cap f(\sigma_1\sqcup\ldots\sqcup\sigma_r)$ is wrong and
should be replaced by $B^d\cap fK=B^d\cap f\st(\sigma_1\sqcup\ldots\sqcup\sigma_r)$; corresponding changes are required in applications of Lemma 8.
(Added later: as in \S\ref{s:prmmw} and in \cite[p. 327, (d)]{Sk06}.)



(b10; added later) In \S4.5, Proof of Lemma 10, First Part, the intersection $\sigma_i\cap\sigma_r$ is a manifold not a ball, so $\partial(\sigma_i\cap\sigma_r)$ is not a strong deformation retract of $\sigma_i\cap\sigma_r-\Int N$.
Thus the argument from the proof of Lemma 12 is not applicable for `retracting from $B^d$ to the ball $N$'.
(Also `proof of Theorem 2' is hard to find in the paper, and when a reader finds it in \S3, he/she sees there no `equation 4'. Presumably equation 4 `in the proof of Lemma 12' not `in the proof of Theorem 2' is meant.)

(b11; added later) In \S4.2, proof of Proposition 30, it is not proved that we may obtain additionally that on every $\sigma_i$ the new $f$  is an embedding isotopic to $f|_{N_i}$, relative to the boundary, for each $i$
\cite[proofs of Propositions 8.4 and 8.5, a generalization of the Whitney trick]{Sk06}.
\end{Remark}


\begin{Remark}\label{r:uw} {\it This is U. Wagner's public response to the criticism of Remark \ref{r:sp} (excluding the `added later' parts)}, only footnotes are mine.
The letter refers to the previous numeration \cite{Sk17o} where (b1)-(b9) of Remark \ref{r:sp} were given as unnumbered bullet points of Remark 3(b).

Dear Arkadiy,

I agree with many of the specific criticisms (all but the last two bullet points in Remark 3(b)).

I agree that these statements and formulations should be made more precise.
(I am not sure regarding the modified dimension restriction you mention in the last
but one bullet points in Remark 3(b), I have to think about this more.)

I think these are helpful and valid criticisms, and we will address them in the next revision.
However, as I tried to explain by skype, I think these things can be easily fixed (in the sense that
no new ideas are needed)  to arrive at a version of the proof that is hopefully
complete according to your definition.\footnote{\label{f:unc} AS: This sentence and the next paragraph are misleading because they convey the author's disagreement with the statement `the proof of \cite{MW16} is {\it incomplete} (in the practical sense described above after Conjecture \ref{t:mmw})',  without explicitly stating disagreement  and so taking responsibility for such a statement.
Instead of explicit disagreement with the {\it incomplete}, the author introduces vague notions of `easily fixed in the sense that no new ideas are needed' and `gaps as opposed to issues of improving the presentation'.
\newline
E.g. working in the smooth category as in
this paper may be regarded as not a new idea (but rather as
`un-introduction' of the new idea of making an argument, originally introduced in the smooth category, work in the PL category).
However, I do not consider
this paper as an `easy fix'.
We also know that {\it realization} of an idea may be harder than its {\it introduction}.
We can learn that there are no `gaps' in a new proof only if the proof has `presentation', which is not
{\it incomplete} (in the above sense), cf. the second paragraph of Remark \ref{r:sp}.a and \cite[Remarks 2.3.abcd]{Sk21d}.}
A number of your remarks can be addressed by straightforward rewording or changing of punctuation.\footnote{\label{f:exp} AS: I welcome these explanations.
However, since the author neither privately distributed the corresponding update of \cite{MW16}, nor submitted it to arxiv, nor gave a list of {\it all} required changes, nor explicitly disagreed with the {\it incompleteness} of the proof, these explanations are not relevant to discussing the {\it incompleteness}.
So I suggested to omit this part of discussion until the revised version will be ready.
\newline
I would also call these explanations `continuing to write a complete proof' rather than `rewording or changing of punctuation'.
Cf. the second paragraph of Remark \ref{r:sp}.a and \cite[Remarks 2.3.abcd]{Sk21d}.
}
In his email from yesterday (February 13), Isaac also gave quite detailed written explanations concerning your remarks; in the following, I repeat some of these explanations, with Isaac's permission:

\begin{itemize}
\item \emph{in Lemmas 19, 20 and 24 ``In the situation given by Lemma NN'' is unclear because
it is not indicated whether the assumption or the conclusion of Lemma NN is meant}

It means under the same hypotheses and using the same notation as in Lemma 10.

\item \emph{in Lemma 16 it is not clear where the restatement of the transversality (the second `i.e.') ends:
at `in $B^d$' or at the end of the display formula;}

it starts at ``i.e.'' and ends at the end of the sentence.

\item \emph{Theorem 7 is not applicable because Proposition 13 does not assert that $\sigma_i\cap \sigma_j$ is a PL manifold.}

$\sigma_i$ and $\sigma_j$ are PL-balls properly embedded inside of a bigger ball, and by general position it is a PL manifold See, e.g., Theorem 1 in Armstrong \& Zeeman, Transversality for PL Manifolds.
\footnote{AS: In the same-day e-mail I recalled the rest of my remark: `this assumption is non-trivial to check, cf. footnote 11 in version 2 of \cite{AMSW}'.
(Added in 2021: the footnote explains problems with using PL transversality unresolved either in \cite{MW16}
or in Isaac's e-mail.)}

\item \emph{in Lemma 16 `we can assume' is way too informal for this delicate part of the proof;}

It means ``after an epsilon-perturbation''

\item \emph{in the proof of Lemma 16 `follows by Theorem 66' is not clear because no block bundle
(required for application of Theorem 66) is given in the statement of Lemma 16.}

The block bundle in question is $\sigma^r \times \varepsilon [-1,1]^{d-s_r}$,
which exists by unknottedness in codimension 3.

\item \emph{in the proof of Lemma 19 `the first property follows from Theorem 70 (presumably meaning Proposition 70) from Appendix A' is not clear because Proposition 70 does not assert any triviality (or the existence of any homeomorphism, which the triviality means).}

The last sentence of Proposition 70 is ``Then any normal bundle over $S^k$ in $M$ is trivial''. This means that any normal (block) bundle over $S^k$ in $M$ (in the sense of Theorem 54) is trivial in the sense of Remark 48.(b) and Def 49. (i.e., the existence of a homeomorphism).
\end{itemize}


For these reasons, I am not convinced that your comments constitute ``gaps'' in our proof
(as opposed to issues of improving the presentation) or that they justify your statement
that our theorems should be considered unproved conjectures.

I understand that you maybe do not consider these explanations sufficient;\footnote{AS: See footnote \ref{f:exp}.}
or maybe your position is (as I believe you said during our skype meeting) that this amounts to
recovering a complete proof, rather than evidence that your comments do not constitute gaps.\footnote{AS: I did not say/write this and I do not understand this explanation of my position.
Sorry for my insufficient knowledge of English.}

From our conversations, I also understand that you consider the words ``easy to fix'', ``gap'',
and ``no new ideas needed'' as vague and not practical. I agree that these it may be hard
to arrive at a general definition of these terms, but I do think that mathematicians often use
these words and reach agreements in specific cases.\footnote{AS: These expressions
are appropriate in a text meant to be vague, but are misleading in a reply to a practical controversial question.
Discussion of a practical question helps to reach an agreement (or an explicit disagreement, which is also valuable).
Bringing into discussion superfluous vague notions helps to convey disagreement,  without taking responsibility for explicitly stating disagreement.}

You suggest a definition of an ``incomplete proof'' in your note.
I am not convinced that your definition is more precise or more practical.\footnote{AS: The above notion of {\it incomplete} is of course not precise.
However, it is explained in terms of practical decisions which mathematicians have regularly to make.}
It is also based on and uses as defining properties social notions of what one
``mathematician should be able to expect from another'', in particular

\begin{itemize}
\item[(2)] ``to recommend, as a referee (based on specific comments) before
recommending publication of results having such a proof;''
\end{itemize}

I am not sure that this serves as a definition of ``incomplete''.
I have certainly both written and received referee reports that recommended revisions without anybody stating
that this meant the proofs were incomplete or the results unproved.\footnote{AS: This is a confusion of a general vague notion of `incomplete' and the practical explanation of `incomplete' above.
Cf. ``No other meaning of `{\it incomplete}' is meant here'' above.
The author can replace the word `incomplete' by `XYZ' in the practical explanation, if it would make it easier to answer the question: {\it do you agree that the proof of the main results in
\cite{MW16} is} XYZ {\it in the above sense?}.}

Since we were not able to reach an agreement on this,
I agree that you should make your criticism public, and we work on a revision.
To me that seems the most productive way forward, and hopefully once the revision is ready we can reach
an agreement regarding the status of our results and proofs.

Best regards,
Uli
\end{Remark}

\comment

, because it gives a bad impression: why rush to publicize partial discussion of comments instead of complete work on a revision

In one point, however, I think I disagree with you:
\footnote{AS: With which exactly statement among written above? }
I see a difference between an arXiv preprint and a final journal version.
Specifically, I think there is value in putting a preprint on the arxiv, even if it is not as polished as one would like, precisely
to get feedback (including from colleagues one does not know who may find the ideas interesting).
\footnote{AS: There is no disagreement because I never stated the opposite.
This remark is misleading because `not as polished as one would like' is vague, see footnote \ref{f:unc},
and because the important practical difference between stating a conjecture and stating a theorem is ignored.
I suggested to replace this paragraph to a statement of whether in view of Remark \ref{r:sp} the author finds
the statement and proof of the main results in \cite{MW16} {\it arxiv-conjectural} (I do).
I call a proof {\it arxiv-conjectural} if one mathematician should be able to expect from another, in a public
arxiv submission, to state a result having such a proof as a conjecture (rather than as a theorem).
No other meaning of `arxiv-conjectural' is meant here.
In particular, I do not mean that at a seminar talk, a result having such a proof should be called a conjecture
(rather than a theorem).}

\endcomment

\begin{Remark}\label{r:uwn}
(a)  An anonymous referee (of version arXiv:1704.00143v2) of this paper misused the anonymous peer review system by presenting criticism whose invalidity becomes clear when its publication is suggested, see \cite[\S6]{Sk21d}.
In particular, the referee stated that the authors of \cite{MW16} do not receive due credit, cf. Remark \ref{r:imm}.b and (b,c) below.

Preliminary versions of the current paper and of \cite{Sk17o} were sent in 2017 before arxiv submissions to
several mathematicians including P. Blagojevi{\'c}, F. Frick, I. Mabillard, S. Melikhov, A. Sz\"ucs, U. Wagner and G. Ziegler.
All criticism any of them wanted to publicly share was presented in arxiv versions.
Some criticism some of them did not want to publicly share.
Later U. Wagner informed me of some private criticism of \cite{Sk17o},
so I asked him to convey my request to send me or to publish this criticism.
None was sent or published.
In (b,c) I present slightly abridged letters to I. Mabillard and U. Wagner,
to which I did not receive any public answers.
{\it This lack of public criticism of \cite{Sk17o} confirms that mathematicians do not have any justification,
up to the standards of a scientific discussion, that the paper arXiv:1704.00143v2 gives insufficient credit to \cite{MW16}.}

(b) {\it Jul 26, 2020.} Dear Uli and Isaac,

I would be grateful if you could either confirm that the attached paper
(which should be the same as arXiv:1704.00143v2) gives due credit to \cite{MW16},
or suggest what could I change there so as to give due credit to \cite{MW16}.
The most relevant part is Remarks 1.4.ab.

I am flexible in terms of credit distribution
(although I am keen on mathematics behind credit distribution).
I do not want to have a priority argument with you.
So I tried to make the attached paper as independent of arXiv:1702.04259 as possible.
E.g. I have the following phrase in Remark 1.4.b:
`In spite of all that I call Theorem 1.2 Metastable Mabillard-Wagner Theorem, in order to concentrate
on mathematics and on reliability standards for research papers, not on priority questions.'
I am willing to give more credits to \cite{MW16} if you think that Remarks 1.4.ab are not proper or insufficient.

I would also be grateful for your comments on the proof in the attached paper.
However, I do understand that this could require much more time than
your recommendations or confirmation concerning due credit to \cite{MW16}.


Best, Arkadiy.

(c) {\it Aug 07, 2020.} Dear Uli, 

We strongly need this discussion to be responsible.   
We do not have time resources to discuss preliminary ideas that do not stand a suggestion to make them public.
Therefore I inform you that each of us can possibly publish any letter on the subject `does arXiv:1704.00143 give proper credit to arXiv:1601.00876v2 ?' --- in particular, this letter and the part of my July 26 letter in Remark 1.4.c.
After presenting the letters we can possibly write whether we agree or disagree, and/or give explanations.
If a part of such a public discussion would become obsolete, we could delete that part (only) by our mutual consent. 

I will only cite in arXiv updates of arXiv:1704.00143v2 this public correspondence, not our emails and Skype discussions to which the math community does not have access and which could be poorly written and less responsible. 

It would be nice to have your reply to my letter of July 26. 
Please just resent your private reply of Aug 7 if you think it is up to the standards of scientific discussion.

Best, A. 
\end{Remark}

}

{\it In this list books, surveys and expository papers are marked by stars}

\end{document}